\documentclass[11pt]{article}
\usepackage{amssymb,amsthm,amsmath,latexsym,tikz, geometry,setspace,enumerate,verbatim,float}
\usepackage[textsize=footnotesize,color=green!40]{todonotes}
\usepackage{changes}
\newtheorem{theorem}{Theorem}

\newtheorem{lemma}[theorem]{Lemma}
\newtheorem{corollary}[theorem]{Corollary}

\newtheorem{observation}[theorem]{Observation}

\newtheorem{claim}{Claim}

\newcommand{\sm}{\setminus}
\newcommand{\cFkk}{\mathcal{F}_{k+1}}
\newcommand{\cFkkk}{\mathcal{F}_{k+2}}
\newcommand{\cFk}{\mathcal{F}_{k}}
\newcommand{\ex}{\text{ex}}
\newcommand{\nex}{\text{nex}}

\newcommand{\Pro}{\mathbb{P}}

\newcommand{\cA}{\mathcal{A}}
\newcommand{\cB}{\mathcal{B}}
\newcommand{\cD}{\mathcal{D}}
\newcommand{\cC}{\mathcal{C}}

\newcommand{\cF}{\mathcal{F}}

\newcommand{\fD}{\mathfrak{D}}

\newcommand{\MR}{\mathbf{MR}}
\newcommand{\JL}{\mathbf{JL}}

\newcommand{\BR}{\mathbf{BR}}

\newcommand{\bP}{\mathbb{P}}
\newcommand{\bE}{\mathbb{E}}

\title{How to determine if a random graph with a fixed degree sequence has a giant component}
\author{Felix Joos\thanks{School of Mathematics, University of Birmingham, Birmingham, B15 2TT, UK. E-mail: f.joos@bham.ac.uk.
This coauthor was supported by the EPSRC, grant no. EP/M009408/1.}\and
Guillem Perarnau\thanks{School of Mathematics, University of Birmingham, Birmingham, B15 2TT, UK. E-mail: p.melliug@gmail.com. This coauthor was supported by the European Research Council, ERC Grant Agreement no. 306349.}\and
Dieter Rautenbach\thanks{Institute of Optimization and Operations Research, Ulm University, Ulm, Germany. E-mail: dieter.rautenbach@uni-ulm.de}\and
Bruce Reed\thanks{CNRS, France; Kawarabayashi Large Graph Project, NII, Japan; IMPA, Brasil; E-mail: breed@sophia.inria.fr}}
\begin{document}

\maketitle

\thispagestyle{empty}

\begin{abstract}
For a fixed degree sequence $\cD=(d_1,...,d_n)$, let $G(\cD)$ be a uniformly chosen (simple) graph on $\{1,...,n\}$ where the vertex $i$ has degree $d_i$.
In this paper we determine whether $G(\cD)$ has a giant component with high probability,
essentially imposing no conditions on $\cD$.
We simply insist that the sum of the degrees in $\cD$ which are not 2 is at least $\lambda(n)$ for some function $\lambda$ going to infinity with $n$. 
This is a relatively minor technical condition, and when $\cD$ does not satisfy it, both the probability that $G(\cD)$ has a giant component and the probability that $G(\cD)$ has no giant component are bounded away from $1$.
\end{abstract}

\section{Introduction}

The traditional Erd\H os-R\'enyi model of  a random network  is of  little use in modelling  the type of complex  networks which modern researchers study. Such a graph can be constructed by adding edges one by one such that in every step, every pair of non-adjacent vertices is equally likely to be connected by the new edge.
However, 21st century networks are of diverse nature and usually exhibit inhomogeneity among their nodes and correlations among their edges.
For example, we observe empirically in the web that certain authoritative pages will have many more links entering them than typical ones. 
This motivates the study, for a fixed degree sequence ${\cal D}=(d_1,...,d_n)$, of a uniformly chosen simple graph $G({\cal D})$ on $\{1,...,n\}$ where the vertex $i$ has degree $d_i$. In this paper, we study the existence of a giant component in $G(\cD)$.

A heuristic argument suggests that  a  giant component will exist provided that the sum of the squares of the degrees
is larger than twice the sum of the degrees.
In 1995,  Molloy and Reed  essentially  proved this to be the case provided that the degree sequence  under consideration satisfied  certain technical conditions~\cite{molloy1995critical}.   This work has attracted considerable attention and has  been applied to   random models of  a wide range of complex networks such as the World Wide Web or biological systems operating at a sub-molecular level~\cite{aiello2000random,albert2002statistical,boccaletti2006complex, newman2003structure,newman2002random}. Furthermore, many authors have obtained related results which formalize the Molloy-Reed heuristic argument under different sets of technical conditions~\cite{bollobas2015old, hatami2012scaling, janson2009new,  kang2008critical, molloy1998size}.

Unfortunately, these technical conditions do not allow the application of such results to many degree sequences that describe real-world networks. While these conditions are of different nature, here we exemplify their limitations with a well-known example, scale-free networks. A network is \emph{scale-free} if its degree distribution follows a power-law, governed by a specific exponent. It is well-known that many real-world networks are scale-free and one of the main research topic in this area is to determine the exponent of a particular network. It has been observed that many scale-free networks have a fat-tailed power-law degree distribution with exponent between $2$ and $3$. This is the case of the World Wide Web, where the exponent is between $2.15$ and $2.2$~\cite{faloutsos1999power}, or the Movie Actor network, with exponent  $2.3$~\cite{barabasi1999emergence}. In scale-free networks with exponents between $2$ and $3$, the vertices of high degree (called \emph{hubs}) have a crucial role in several of the network properties such as in the ``small-world'' phenomenon. However, one of the many technical conditions under the previous results on the existence of a giant component in $G(\cD)$ hold, is that the vertices of high degree do not have a large impact on the structure of the graph. (In particular, it is required that there is no mass of edges in vertices of non-constant degree.) Hence, often these results cannot be directly applied to real-world networks where hubs are present and for each particular network ad-hoc approaches are needed (see e.g. the Aiello-Chung-Lu model for the case of scale-free networks~\cite{aiello2000random}).

Another problem is that all the previous results apply to a sequence of degree sequences $(\cD_n)_{n\geq 1}$ satisfying various technical conditions  instead of a  single degree sequence $\cD$. 
Finally, all the previous results on the existence of a giant component in $G(\cD)$ do not cover degree sequences where most of the vertices have degree $2$. 
%While this might seem a degenerated case, it can be observed in several networks that model real-world systems such as the rail network.

In this paper we characterize when $G(\cD)$ has a giant component for \emph{every} feasible\footnote{A degree sequence is feasible if there is a graph with the given degree sequence.}  degree sequence $\cD$ of length $n$. We only require that the sum of the degrees in the sequence which are not $2$ is at least $\lambda(n)$ for some arbitrary function $\lambda$ going to infinity with $n$. Besides the fact that it is a relatively minor technical condition, we also show that if it is not satisfied, both the probability that $G(\cD)$ has a giant component and the probability that $G(\cD)$ has no giant component are bounded away from $1$.

It turns out that the heuristic argument which was used in~\cite{molloy1995critical} to describe the existence of a giant component in $G(\cD)$ for degree sequences satisfying some technical conditions and that was generalized in the subsequent papers~\cite{bollobas2015old, hatami2012scaling, janson2009new,  kang2008critical, molloy1998size}, actually suggests the wrong answer for general degree sequences.
Precisely, if we let $S$ be a smallest set such that (i) no vertex outside of $S$ has degree bigger than a vertex in $S$, and
(ii) the sum of the  squares of the degrees of the vertices  outside of  $S$  is at most twice the sum of their degrees, then whether
or not a giant component exists depends on the sum of the degrees of the vertices in $S$, not on the sum of the squares
of the degrees of the vertices in $S$ as suggested by this heuristic argument%\footnote{Although the set $S$ is not unique here, the multiset of the degrees of its vertices is.}.

This new unified criterion on the existence of a giant component in $G(\cD)$ is valid for every sequence $\cD$ and implies all the previous results on the topic both for arbitrary degree sequences~\cite{bollobas2015old, janson2009new, molloy1995critical} or for particular models~\cite{aiello2000random}.

\subsection{The Molloy-Reed Approach}\label{subsec:MR}
Let us first describe the result of Molloy and Reed~\cite{molloy1995critical}. For every $1\leq i\leq n$, one  can explore the component containing a specific initial  vertex $i$ of a graph
on $\{1,...,n\}$ via breadth-first search. Initially we have $d_i$ ``open'' edges out of $i$. Upon  exposing the other endpoint $j$ of such an open edge, it is no longer open, but we
gain $d_j-1$ open edges out of $j$. Thus, the number of open edges
has increased by $d_j-2$ (note that this is negative if $d_j=1$).

One  can generate the random graph  $G({\cal D})$ for ${\cal D}=(d_1,...,d_n)$
and carry out this exploration at the same time,  by choosing  each vertex  as $j$ with the appropriate
probability.

Intuitively speaking, the probability we pick a specific vertex $j$ as the other endpoint of the first exposed edge is proportional to
its degree. So, the expected increase in the number of open edges in the first step  is equal to
$ \frac{ \sum_{k\in \{1,..,n\}, k \neq i} d_k(d_k-2)}{\sum_{k\in \{1,..,n\}, k \neq i} d_k}$.
Thus, it is positive essentially if and only if the sum of the squares of the degrees exceeds twice the sum of the degrees.

Suppose that this expected increase remains the same until we have exposed
 a linear number of vertices. It seems intuitively clear that if the
 expected increase  is less than $0$, then the probability that initial vertex $i$ is in a linear order component is $o(1)$, and hence the probability that $G({\cal D})$ has no linear order component is $1-o(1)$.
If for some positive constant $\epsilon$, the expected increase is at least $\epsilon$,
then there is some $\gamma=\gamma(\epsilon)>0$ such that
the probability that $i$ is in a component with at least $\gamma n$ vertices exceeds $\gamma$.

 In~\cite{molloy1995critical}, Molloy and Reed proved, subject to certain technical conditions
 which required them to discuss sequences of degree sequences rather than one single degree  sequence, that we essentially have that 
 (i) if $\frac{ \sum_{k=1}^n d_k(d_k-2)}{\sum_{k=1}^{n} d_k} >\epsilon$ for some $\epsilon>0$, then the probability that $G({\cal D})$ has a	
 giant component is $1-o(1)$, and 
 (ii) if $\frac{ \sum_{k=1}^n d_k(d_k-2)}{\sum_{k=1}^{n} d_k} <-\epsilon $ for some $\epsilon>0$,  then the probability that $G({\cal D})$ has no giant component is $1-o(1)$.
 We present their precise result and some of its generalizations later in this introductory section.

 \subsection{Our Refinement}

It turns out that,  absent  the imposed  technical conditions, the expected increase may change drastically during the exploration process. Consider for example the situation in which  $n=k^2$ for some large odd integer $k$, $d_1=d_2=..=d_{n-1}=1$ and $d_n=2k$. Then $\frac{ \sum_{k=1}^n d_k(d_k-2)}{\sum_{k=1}^{n} d_k} =\frac{4k^2-4k-(n-1)}{2k+n-1}\approx3$, and so the Molloy-Reed approach would suggest that with probability $1-o(1)$ there is a giant component. However, with probability $1$, $G({\cal D})$  is  the disjoint union of a star with $2k$ leaves and
$\frac{n-2k-1}{2}$ components of order $2$ and hence it has no giant component. The problem is that as soon as we explore vertex $n$,
the expected increase drops from roughly $3$ to $-1$, so it does not stay
positive throughout (the beginning of) the process.

Thus, we see that the Molloy-Reed criterion cannot be extended for general degree sequences.
To find a variant which applies to arbitrary degree sequences, we need to characterize those for which the expected increase remains positive for a sufficiently long time.

Intuitively, since the probability that we explore a vertex is essentially proportional to its degree, in lower bounding the length of the period during which the expected increase remains positive, we could assume that the exploration process picks at each step a highest degree vertex that has not been explored yet.
Moreover, note that vertices of degree $2$ have a neutral role in the exploration process as exposing such a vertex does not change the number of open edges, provided we assume that our component locally looks like a tree (which turns out to be a good approximation around the critical window).
These observations suggest that we should focus on the following  invariants of $\cD$ defined by considering a permutation $\pi$ of the vertices that satisfies $d_{\pi_1} \le ... \le d_{\pi_n}$:
\begin{itemize}
\item[-] $j_{\cal D} = \min \left(\Big\{j :j\in [n]\mbox{ and }\sum\limits_{i=1}^{j}d_{\pi_i}(d_{\pi_i}-2)> 0\Big\}\cup \Big\{ n\Big\}\right)$,
\item[-] $R_{\cal D}=\sum\limits_{i=j_{\cal D}}^{n}d_{\pi_i}$, and
\item[-] $M_{\cal D}=\sum\limits_{i\in [n]:\,d_i\neq 2}d_{i}$.
\end{itemize}
We emphasize that these invariants are determined by the multiset of the degrees given by ${\cal D}$ and are independent from the choice of $\pi$ subject to $d_{\pi_1} \le ... \le d_{\pi_n}$.

Our intuition further suggests that in the exploration process, the expected increase
in the number of open edges will be positive until we have explored
$R_{\cal D}$ edges and will then become negative. Thus, we might expect to
explore a component with about $R_{\cal D}$ edges, and indeed we can show this is the case.

This allows us to prove our main result which is that whether $G({\cal D})$ has a giant component
essentially depends on whether $R_{\cal D}$ is of the same order as $M_{\cal D}$ or not.
There is however a caveat,  this is not true if essentially all vertices have degree $2$.

For any function $\lambda:\mathbb{N}\to \mathbb{N}$, we say a degree sequence ${\cal D}$  is {\it $\lambda$-well-behaved}  or simply {\it well-behaved} if $M_{\cal D}$ is at least $\lambda(n)$. Our main results hold for any function $\lambda\to \infty$ as $n\to \infty$.
\begin{theorem}
\label{maintheorem1}
For any function $\delta\to 0$ as $n\to\infty$, for every $\gamma>0$, if ${\cal D}$ is a well-behaved  feasible degree sequence with
 $R_{\cal D} \le \delta(n)M_{\cal D}$, then the  probability that $G({\cal D})$ has a  component
 of order at least  $\gamma n$ is $o(1)$.
\end{theorem}

\begin{theorem}
\label{maintheorem2}
For any positive constant $\epsilon$, there is a $\gamma>0$,
such that if ${\cal D}$ is a well-behaved feasible degree sequence with  $R_{\cal D} \ge \epsilon M_{\cal D}$,
then the probability that $G({\cal D})$ has a component  of order at least  $\gamma n$ is  $1-o(1)$.
\end{theorem}

As we shall see momentarily, previous results in this field apply to sequences of degree sequences, and required that both 
the proportion of  elements of a given degree and the average  degree of an element of the degree sequences in the sequence approaches a limit in some smooth way. We can easily deduce results for {\it every} sequence of  (feasible) degree sequences
from the two theorems above, and from our results on degree sequences which are not well-behaved, presented in the next section.

We denote by $\fD=({\cal D}_n)_{n\geq 1}$ a sequence of degree sequences where $\cD_n$ has length $n$.  We say that $\fD$ is {\it well-behaved} if for every $b$, there is an $n_b$ such that for all $n>n_b$, we have $M_{{\cal D}_n}>b$;  $\fD$
 is {\it lower bounded} if for some $\epsilon>0$, there is an $n_\epsilon$ 
such that for all $n>{n_\epsilon}$, we have $R_{{\cal D}_n} \ge \epsilon M_{{\cal D}_n}$; and  
$\fD$ is {\it upper  bounded} if for every  $\epsilon>0$, there is an $n_\epsilon$ 
such that for all $n>{n_\epsilon}$, we have $R_{{\cal D}_n} \le \epsilon M_{{\cal D}_n}$. 

The following is an immediate consequence of Theorem \ref{maintheorem1} and  \ref{maintheorem2}, and Theorem~\ref{deg2stheorem2}, which will be presented in the next section.

\begin{theorem}
\label{sequencetheorem}
For any  well-behaved  lower bounded sequence of feasible degree sequences $\fD$, there is a $\gamma>0$ such that the probability that $G({\cal D}_n)$ has a component of order at least  $\gamma n$ is $1-o(1)$.

For any  well-behaved  upper bounded sequence of feasible degree sequences $\fD$ and every  $\gamma>0$, 
the probability that $G({\cal D}_n)$ has a component of order  at least $\gamma n$ is $o(1)$.

If a sequence of  feasible degree sequences $\fD$ is either not well behaved or neither upper bounded nor lower bounded, then for every sufficiently small positive  $\gamma$, there is a $0<\delta<1$ such that there are  both arbitrarily large $n$ for which the probability that $G({\cal D}_n)$ has a component of order at least $\gamma n$ is at least  $\delta$, and arbitrarily large $n$ for which the probability that $G({\cal D}_n)$ has a component of order at least $\gamma n$ is at most  $1-\delta$. 
\end{theorem}

\subsection{The Special Role of Vertices of Degree 2}\label{subsec:many 2}

At first glance, it may be surprising that the existence of a giant component
depends  on  the ratio between $R_{\cal D}$ and $M_{\cal D}$ rather than the ratio between	
$R_{\cal D}$ and $\sum_{i=1}^nd_i$. It may also be unclear why we have to treat
differently degree sequences where the sum of the degrees which are not 2 is bounded.

To clarify why our results are stated as they are, we now highlight the special role of vertices of degree 2.

The non-cyclic components\footnote{A component is \emph{cyclic} if it is a cycle and \emph{non-cyclic} if it is not.} of $G({\cal D})$ can be obtained by subdividing the edges of a multigraph $H({\cal D})$ none of whose vertices have degree 2  so that
every loop is subdivided at least twice, and all but at most one edge of every set of parallel edges is subdivided at least once.
Note that $H(\cD)$ can be obtained from $G(\cD)$ by deleting all cyclic components and suppressing all vertices of degree~$2$.\footnote{Here and throughout the paper, when we say we suppress a vertex $u$ of degree~$2$, this means we delete $u$ and we add an edge between its neighbours. Observe that this may create loops and multiple edges, so the resulting object might not be a simple graph.
}
Clearly, $H({\cal D})$ is uniquely  determined by $G({\cal D})$. Moreover, the degree sequence of $H(\cD)$ is precisely that of $G(\cD)$ without the vertices of degree $2$.

The number of vertices of a  non-cyclic component of $G({\cal D})$ equals the sum of the number of vertices of the corresponding component of
$H({\cal D})$ and the number of vertices of degree 2 used in subdividing its edges.
Intuitively, the second term in this sum  depends
on the proportion of the edges  in the   corresponding component of $H({\cal D})$.
Subject to the caveat mentioned above and discussed below, if the number of vertices of degree 2  in $G(\cD)$  is much larger than the size\footnote{As it is standard, we use \emph{order} and \emph{size} to denote the number of vertices and the number of edges of a graph, respectively.} of $H({\cal D})$, then  the probability that  $G({\cal D})$ has
a giant component   is essentially the same as the probability   $H({\cal D})$ has a component containing a positive fraction of its edges.
The same is true, although not as immediately obvious, even if the number of vertices of degree 2 is not this large.
\begin{theorem}
\label{deg2stheorem1}
For every $\gamma>0$, there exists a $\rho>0$ such that for every well-behaved feasible degree sequence ${\cal D}$, the probability that $G({\cal D})$  has a component of order at least $\gamma n$ and $H({\cal D})$ has no component of size at least $\rho M_{\cal D}$ is $o(1)$.
\end{theorem}

\begin{theorem}
\label{deg2stheorem}
For every  $\rho>0$, there exists a $\gamma>0$ such that for
every well-behaved feasible degree sequence ${\cal D}$,  the probability that $G({\cal D})$ has no component of order at least $\gamma n$ and $H({\cal D})$  has a component of size at least $\rho M_{\cal D}$ is $o(1)$.
\end{theorem}

As we mentioned above, degree sequences which are not well behaved 
behave differently from well-behaved degree sequences.  For instance,  suppose that $M_{\cal D}=0$. Then $H({\cal D})$ is empty  and $G({\cal D})$ is a uniformly chosen disjoint union of cycles. In this case it is known that the probability of having a giant component is bounded away both from $0$ and $1$.
Indeed, the latter statement also holds whenever $M_{\cal D}$ is at most $b$ for any constant $b$.
\begin{theorem}
\label{deg2stheorem2}
For every $b\geq 0$ and every $0<\gamma <\frac{1}{8}$, 
there exist a positive integer $n_{b,\gamma}$ and a $0<\delta<1$
such that if $n>n_{b,\gamma}$ and ${\cal D}$ is a degree sequence with $M_{\cal D}\leq b$,
then the probability that there is a component of order at least
$\gamma n$ in $G({\cal D})$ lies between $\delta$ and $1-\delta$.
\end{theorem}

This theorem both explains why we concentrate on well-behaved degree sequences and sets
out how degree sequences which are  not well-behaved actually behave (badly obviously).
Combining it with  Theorem \ref{maintheorem1} and \ref{maintheorem2},
it immediately implies Theorem \ref{sequencetheorem}. We omit the straightforward details.

\subsection{Previous Results}\label{sec:previous results}

The study of the existence of a giant component in random graphs with an arbitrary  prescribed degree sequence\footnote{Random graphs with special degree sequences had been studied earlier (see, e.g.~\cite{luczak1989sparse,Wormald80}).}, started with the result of Molloy and Reed~\cite{molloy1995critical}. Although they define the concept of asymptotic degree sequences, we will state all the previous results in terms of sequences of degree sequences $\fD=(\cD_n)_{n\geq 1}$. Using a symmetry argument, one can easily translate results for sequences of degree sequences to asymptotic degree sequences, and vice versa. For every $\cD_n=(d_1^{(n)},\dots, d_n^{(n)})$, we define $n_i=n_i(n)=|\{j\in [n]:\, d_j^{(n)}=i\}|$. 
Recall that we only consider degree sequences such that $n_0=0$.
   
Before stating their result, we need to introduce a number of properties of sequences of degree sequences.
A sequence of degree sequences $\fD$ is
\begin{itemize}
\item[-] \textit{feasible}, if for every $n\geq 1$, there exists at least one simple graph on $n$ vertices with degree sequence $\cD_n$.
\item[-] \textit{smooth}, if for every nonnegative integer $i\geq 0$, there exists $\lambda_i\in[0,1]$ such that $\lim_{n\to\infty} \frac{n_i}{n}=\lambda_i$.
\item[-] \textit{sparse}, if there exists  $\lambda\in (0,\infty)$ such that $\lim_{n\to \infty} \sum_{i\geq 1} \frac{in_i}{n}=\lambda$.
\item[-] \textit{$f$-bounded}, for some function $f: \mathbb{N} \to \mathbb{R}$, if $n_i=0$ for every $i> f(n)$.
\end{itemize}
In particular, observe that random graphs $G(\cD_n)$ arising from a sparse sequence of degree sequences $\fD$ have a linear number of edges, provided that $n$ is large enough.

\noindent Given that $\fD$ is smooth, we define the following parameter
\begin{align*}
Q(\fD)= \sum_{i\geq 1} i(i-2)\lambda_i\;.
\end{align*}
Note that $Q(\fD)$ is very close to the notion of initial expected increase described in Section~\ref{subsec:MR}.

We say a sequence of degree sequences $\fD$ satisfies the \textit{$\MR$-conditions} if
\begin{enumerate}[a.1)]
 \item[(a.1)] it is feasible, smooth and sparse,
 \item[(a.2)]\label{cond:a2} it is $n^{1/4-\epsilon}$-bounded, for some $\epsilon>0$,
 \item[(a.3)]\label{cond:a3} for every $i\geq 1$, $\frac{i(i-2)n_i}{n}$ converges uniformly to $i(i-2)\lambda_i$, and
 \item[(a.4)]\label{cond:a4} $\lim_{n\to\infty} \sum_{i\geq 1} i(i-2)\frac{n_i	}{n}$ exists and converges uniformly to $\sum_{i\geq 1} i(i-2)\lambda_i$.
\end{enumerate}
For a precise statement of the uniform convergence on conditions~(a.3)--(a.4), we refer the reader to~\cite{molloy1995critical}. Note that these conditions imply that $\lambda=\sum_{i\geq 1} i \lambda_i$.

Now we can precisely state the result of Molloy and Reed~\cite{molloy1995critical}.
\begin{theorem}[Molloy and Reed~\cite{molloy1995critical}]\label{thm:MR}
 Let $\fD=(\cD_n)_{n\geq 1}$ be a sequence of degree sequences that satisfies the $\MR$-conditions. Then,
 \begin{enumerate}
  \item if $Q(\fD)>0$, then there exists a constant $c_1>0$ such that the probability that $G(\cD_n)$ has a component of order at least $c_1 n$ is $1-o(1)$.
  \item if $Q(\fD)<0$ and the sequence is $n^{1/8-\epsilon}$-bounded for some $\epsilon>0$, then for every constant $c_2>0$,  the probability that $G(\cD_n)$ has no component of order at least $c_2 n$ is $1-o(1)$.
 \end{enumerate}
\end{theorem}
\noindent Note that the case $\lambda_2=1$ is not considered in Theorem~\ref{thm:MR}, since $\lambda_2=1$  implies $Q(\fD)=0$.

Theorem~\ref{thm:MR} has been generalized to other sequences of degree sequences, which in particular include the case $Q(\fD)=0$.  In Section \ref{sec.imp}, we show that Theorem~\ref{sequencetheorem} implies all the criteria for the existence of a giant component in $G(\cD_n)$ introduced below.\footnote{Note that some of these results give a more precise description on the order of the largest component. Our results only deal with the existential question.} 

We say a sequence of degree sequences $\fD$ satisfies the \textit{$\JL$-conditions} if
\begin{enumerate}[(b.1)]
 \item is feasible, smooth, and sparse,
 \item \label{cond:b2} $\sum_{i\geq 1} i^2 n_i=O(n)$, and
 \item $\lambda_1>0$.
\end{enumerate}
Observe that if $\fD$ satisfies the $\JL$-conditions, then, by~(b.\ref{cond:b2}), it is also $O(n^{1/2})$-bounded. Moreover, they also imply that $\lambda=\sum_{i\geq 1} i \lambda_i$.
Janson and Luczak in~\cite{janson2009new} showed that one can prove a variant of Theorem~\ref{thm:MR} obtained by replacing the $\MR$-conditions by the  $\JL$-conditions.\footnote{Their result gives convergence in probability of the proportion of vertices in the giant component and they also consider the case $Q(\fD)=0$.} 
They also note that if $\lambda_2=1$, then the criterion based on $Q(\fD)$ does not apply. Our results completely describe the case $\lambda_2=1$. 

We say a sequence of degree sequence $\fD$ satisfies the \textit{$\BR$-conditions} if
\begin{enumerate}[(c.1)]
 \item it is feasible, smooth and sparse, 
 \item $\sum_{i\geq 3}\lambda_i>0$, and
 \item $\lambda=\sum_{i\geq 1} i \lambda_i$.
\end{enumerate}

Bollob\'as and Riordan in~\cite{bollobas2015old} proved a version of Theorem~\ref{thm:MR} for sequences of degree sequences obtained by replacing the $\MR$-conditions by the $\BR$-conditions.\footnote{They also proved some results on the distribution of the  order of the largest component and also consider the case $Q(\fD)=0$.} 
\medskip

Theorem~\ref{thm:MR} and its extensions provide easy-to-use criteria for the existence of a giant component and have been widely used by many researchers in the area of complex networks~\cite{albert2002statistical,boccaletti2006complex,newman2003structure}.
However, the technical conditions on $\fD$ to which they can be applied, restrict its applicability, seem to be artificial and are only required due to the nature of the proofs.
It turns out to be the case that many real-world networks do not satisfy these conditions.
For this reason, researchers have developed both ad-hoc approaches for proving 
results for specific types of degree sequences and variants of the Molloy-Reed result which require different sets of technical conditions to be satisfied. 

An early example of an ad-hoc approach is the work of Aiello, Chung and Lu on Power-Law Random Graphs~\cite{aiello2000random}.  They introduce a model depending on two parameters $\alpha,\beta>0$ that define a degree sequence satisfying $n_i=\lfloor e^{\alpha} i^{-\beta}\rfloor$. One should think about these parameters as follows: $\alpha$ is typically large and determines the order of the graph (we always have $\alpha=\Theta(\log{n})$), and $\beta$ is a fixed constant that determines the power-decay of the degree distribution. Among other results, the authors prove that there exists $\beta_0>0$, such that if $\beta>\beta_0$ the probability that there is a component of linear order is $o(1)$ and if $\beta<\beta_0$ the probability there is a component of linear order is $1-o(1)$. Here, the previous conditions are only satisfied for certain values of $\beta$ and the authors need to do additional work  to determine when a giant component exists for other values of $\beta$. In Section~\ref{sec.imp} we will show how Theorem~\ref{sequencetheorem} also implies Aiello-Chung-Lu results on the existence of a giant component in the model of Power-Law Random Graphs.

\subsection{Future Directions}
Beginning with the early results of Molloy and Reed, the study of the giant component in random graphs with prescribed degree sequence has attracted a lot of attention. Directions of study include determining the asymptotic order of the largest component in the subcritical regime or estimating the order of the second largest component in both regimes~\cite{bollobas2015old, hatami2012scaling, janson2009new, joseph2010component, kang2008critical, molloy1998size, riordan2012phase}. It would be interesting to extend these known results to arbitrary degree sequences.

For example, Theorem~\ref{maintheorem1} and~\ref{maintheorem2} precisely describe the appearance of a giant component when the degree sequence is well-behaved.
While bounds on the constant $\gamma$ in terms of $\delta$ and $\epsilon$ respectively,
may follow from their respective proofs,
these bounds are probably not of the right order of magnitude. 
Molloy and Reed in~\cite{molloy1998size}, precisely determined this dependence for sequences of degree sequences that satisfy the $\MR$-conditions. Precise constants are also given in~\cite{bollobas2015old,fountoulakis2009general,janson2009new}. 
We wonder whether it is possible to determine the precise dependence on the parameters for arbitrary degree sequences.
It is likely that our methods can be used to find this dependence and to determine the order of the second largest component when a giant one exists.

Another direction is the study of site and bond percolation in $G(\cD)$ for arbitrary degree sequences $\cD$. This problem has been already approached for sequences of degree sequences that satisfy certain conditions similar to the ones presented in Section~\ref{sec:previous results}~\cite{bollobas2015old,fountoulakis2007percolation,janson2008percolation,riordan2012phase}. 
%\felix{I suggest to delete the following sentence. To which question do you actually refer to?}
%In particular, our approach might be useful to answer a question of Nachmias and Peres on the percolation threshold for random $d$-regular graphs when $d\to \infty$ as $n\to \infty$~\cite{nachmias2010critical}.

Motivated by some applications in peer-to-peer networks~(see, e.g.~\cite{bourassa2003swan}), one can study efficient sampling of the random graph $G(\cD)$. Cooper et al.~\cite{cooper2007sampling} showed that the switching chain rapidly mixes for $d$-regular graphs for every $3\leq d\leq n-1$.
Greenhill~\cite{greenhill2015switch} recently extended this result to $G(\cD)$, but, due to some technical reasons, this result only holds if the maximum degree in $\cD$ is small enough.

Many other basic properties of $G(\cD)$, such as determining its diameter~\cite{fernholz2007diameter,van2005distances,van2005distances2} or the existence of giant cores \cite{cooper2004cores,fernholz2004cores,janson2007simple}, have already been studied for certain sequences of degree sequences. We believe that our method can help to extend these results to arbitrary degree sequences.

\section{A Proof Sketch}

\subsection{The Approach}

The proofs of Theorem~\ref{deg2stheorem1}, \ref{deg2stheorem} and  \ref{deg2stheorem2} are simpler than the remaining proofs and
we delay any discussion of these results to Section~\ref{deg2section} and~\ref{sec:relating}. 
By applying them, we see that in order to prove Theorem \ref{maintheorem1} and \ref{maintheorem2} it is enough to prove the following results:

\begin{theorem}
\label{thm: main3}
For any function $\delta\to 0$ as $n\to\infty$ and for every $\gamma>0$,
% \new{such that} 
if $\cD$ is a well-behaved degree sequence with
 $R_{\cD} \le \delta(n)M_{\cD}$, then the   probability that $H(\cD)$ has a  component
 of size at least $\gamma M_{\cD}$ is $o(1)$.
\end{theorem}

\begin{theorem}
\label{thm: main4}
For any positive constant $\epsilon$, there is a $\gamma>0$ such that if $\cD$ is a well-behaved degree sequence with  $R_{\cD} \ge \epsilon M_{\cD}$, then the
probability that $H(\cD)$ has a component  of size  at least $\gamma M_\cD$ is   $1-o(1)$.
\end{theorem}

The proofs of both theorems analyse an exploration process similar to the one discussed in Section~\ref{subsec:MR} by combining probabilistic tools with a combinatorial switching argument.
However, we  will focus on the edges  of  $H(\cD)$ rather than the ones of $G({\cal D})$.
Again, we will need to bound the expected increase of the number of open edges throughout the process and prove that the (random) increase is highly concentrated around its expected value.
In order to do so, we will need to bound the probability that the next vertex of $H({\cal D})$ explored in the process, is a specific vertex $w$.  
One of the key applications of our combinatorial switching technique will be to estimate this probability and show that it is approximately proportional to the degree of $w$.

Crucial to this approach is that the degrees of the vertices explored throughout the process are not too high. 
Standard arguments for proving concentration of a random variable  require that the change at each step is relatively small. 
This translates precisely to an upper bound on the maximum degree of the explored graph. 
Furthermore, without such a bound on the maximum degree, we cannot obtain good bounds on the probability that a certain vertex  $w$ is the next vertex explored in the process. 
So, a second key ingredient in our proofs will be a preprocessing step which allows us to handle the vertices of high degree, 
ensuring that we will not encounter them in our exploration process.

\subsection{The Exploration Process}

We consider a variant of the exploration process where we
start our exploration at a non-empty set $S_0$ of vertices of $H(\cD)$, rather than at just one vertex.

Thus, we see that
the exploration takes $|V(H(\cD))\sm S_0|$ steps and produces sets
$$
S_0\subset S_1\subset S_2\subset \ldots \subset S_{|V(H(\cD))\sm S_0|}\;,
$$
where $w_t=S_t\sm S_{t-1}$ is either a neighbour of a vertex
$v_t$ of $S_{t-1}$ or is a randomly chosen vertex in $V(H(\cD))\sm S_{t-1}$ if there are no edges between $S_{t-1}$ and $V(H(\cD))\sm S_{t-1}$.

%Thus, letting $n'$ be the number of vertices of $H({\cal D})$, we see that
%the exploration takes $n'-|S_0|$ steps and produces sets
%$$S_0\subset S_1\subset S_2\subset \ldots \subset S_{n'-|S_0|}$$
%where $w_t=S_t\sm S_{t-1}$ is either a neighbour of a vertex
%$v_t$ of $S_{t-1}$ or is a randomly chosen vertex if there are no edges out of $S_{t-1}$.

To specify our exploration process precisely, we need to describe how we choose $v_t$
and $w_t$. To aid in this process,  for each vertex $v\in V(H({\cal D}))$ we will choose a uniformly random permutation of its adjacency list in $G({\cal D})$. 
For this purpose, an \emph{input} of our exploration process consists of a graph $G$ equipped with an ordering of its adjacency lists for all vertices $v\in V(H(\cD))$. 
Applying the method of deferred decisions (cf. Section 2.4 in \cite{bookmolloyreed}),
we can generate these random linear orders as we go along with our process.
We note that this yields, in a natural manner, an ordering of the non-loop edges
of $H({\cal D})$ which have the vertex $v$ as an endpoint.
If there are no edges between $S_{t-1}$ and $V(H(\cD))\sm S_{t-1}$, we choose each vertex of $V(H({\cal D}))\sm S_{t-1}$ to be $w_t$ with probability proportional to its degree.
Otherwise we choose the smallest vertex $v_t$ of $S_{t-1}$ (with respect to the natural order in $\{1,\dots,n\}$), which has a neighbour in $V(H(\cD))\sm S_{t-1}$.
We expose  the edge of $H(\cD)$ from $v_t$
to $V(H(\cD))\sm S_{t-1}$ which appears first in our random ordering and let $w_t$ be its other endpoint.
Furthermore, we expose all the edges of $H(\cD)$ from  $w_t$ to $S_{t-1}\sm \{v_t\}$ as well as the loops incident to $w_t$. 
Finally, we expose
the paths of $G({\cal D})$ corresponding to the edges of $H({\cal D})$ which we have just exposed and the position in the random permutation of the adjacency list of 
$w_t$ in $G({\cal D})$ of the edges we have just exposed.  

Thus, after $t$ iterations of our exploration process
we have exposed
\begin{itemize}
\item the subgraph of $H({\cal D})$ induced by $S_t$,
\item  the paths of $G({\cal D})$
corresponding to the exposed edges of $H({\cal D})$, and
\item where each  initial and final edge of  such a path  appears in the random permutation of the adjacency list of any of its  endpoints which is also an  endpoint of the path.
\end{itemize}

We refer to this set of information as the {\it configuration ${\cal C}_t$ at time $t$}. A configuration can also be understood as a set of inputs.
During our analysis of the exploration process,
we will consider all the probabilities of events conditional on the current configuration.

An important parameter for our exploration process
is the number $X_t$ of edges of $H({\cal D})$ between $S_t$ and $V(H({\cal D}))\sm {S}_t$.
We note that if $X_t=0$, then $S_t$ is the union of some components of $H({\cal D})$
containing all of $S_0$. 
We note that if $|S_0|=1$, then every $X_t$  is a lower bound on the
 maximum size of a component of $H({\cal D})$ (not necessarily the one containing the vertex in $S_0$).

We  prove Theorem \ref{thm: main3} by showing that under its  hypotheses for every vertex $v$ of $H({\cal D})$, 
there is a set $S_0=S_0(v)$  containing $v$ such that, 
given we start our exploration process with $S_0$, the probability that  there is a $t$
with $X_t=0$ for which the number of edges within $S_t$ is at most $\gamma M_{\cal D}$, is $1-o(M_{\cal D}^{-1})$.
Since $H({\cal D})$ has at most $2M_{\cal D}$ vertices, it follows that the probability that
$H({\cal D})$ has a component of size at least $\gamma M_{\cal D}$ is $o(1)$.
The set $S_0\sm\{v\}$ is a set of highest degree vertices the sum of whose degrees
exceeds $R_\cD$.
By the definition of $j_{\cal D}$ and $R_{\cal D}$, this implies that, unless $X_0=0$,  the expectation of $X_1-X_0$ is negative.
We show that, as the process continues, the expectation of $X_t-X_{t-1}$  becomes even smaller.
We can prove that the actual change of $X_t$ is highly concentrated around its expectation and hence  complete the proof, 
because $S_0$ contains all the high degree
vertices and so in the analysis of our exploration process we only
have to deal with low degree vertices.

We  prove  (a slight strengthening of)  Theorem \ref{thm: main4}  for graphs without large degree vertices by showing that under its hypotheses and setting $S_0$ to be a random vertex $v$ chosen with probability proportional to its degree, with probability $1-o(1)$, there exists some $t$ such that $X_t\geq\gamma M_{\cal D}$ (and hence there is a component of $H({\cal D})$ of size at least $\gamma M_{\cal D}$).
Key to doing so  is that the expected increase of $X_t$  is a positive fraction of the increase in the sum of the degrees of the vertices in $S_t$  until this sum approaches $R_{\cal D}$.  
To handle the high degree vertices, we expose the edges whose endpoints are in components containing a high degree vertex.
If this number of edges is at least a constant fraction of $M_{\cal D}$, then we can show that in fact all the high degree vertices lie in one component, which therefore contains a constant fraction of the edges of $H({\cal D})$.
Otherwise, we  show that the conditions of Theorem   \ref{thm: main4} (slightly relaxed) hold  in the remainder of the graph, which has no high degree vertices, 
so we can apply  (a slight strengthening of)  Theorem \ref{thm: main4} to find the desired component of $H({\cal D})$.

%
%  if there are no large degree vertices then
% setting $S_0$ to be a random vertex chosen where the probability we pick $v$ is
%proportional to its degree,  with probability $1-o(1)$, some $X_t$ is  at least $\gamma M_{\cal D}$ (and hence
%there is a component of $H({\cal D})$ of  size $\gamma M_{\cal D}$). Key to
%doing so  is that the expected increase in $X_t$  is a positive fraction
%of the increase in the sum of the degrees of the vertices in $S_t$  until this sum approaches $R_{\cal D}$.  To handle the high degree vertices, we expose the sum of the degrees of the vertices in those components containing a high degree vertex.
%If this is at least a constant fraction of $M_{\cal D}$,
%then we can show that in fact all the high degree vertices lie in one component,
%which therefore contains a constant fraction of the edges of $H({\cal D})$.
%Otherwise, we can apply our exploration process in the remainder of the graph to find the desired component of $H({\cal D})$.

\section{Switching}\label{sec:swi}

As mentioned above, the key to extending our branching analysis to arbitrary well-behaved degree sequences is a combinatorial switching argument. In this section, 
we describe the type of switchings we consider and demonstrate the power of the technique. 

Let $H$ be a multigraph.
We say a multigraph $H'$ is obtained by  \emph{switching} from $H$
on  a pair of orientated and distinct
%orientations of distinct  
edges $uv$ and $xy$
if $H'$ can be obtained from $H$ by deleting  $uv$ and $xy$
as well as adding the edges $ux$ and $vy$. 
Observe that switching $ux$ and $vy$ in $H'$ yields $H$. Observe further that if  $H$ is simple and we want to ensure that $H'$  is simple, then we must   insist that $u \ne x$, $v \ne y$ and, unless $u=y$ or $v=x$, the edges $ux$ and $vy$ are not edges of $H$.

Switching  was  introduced in the late 19th century  by Petersen~\cite{Pet91}.
Much later, McKay~\cite{McK81} reintroduced the method to count graphs  with prescribed degree sequences and, together with Wormald,  used it in the study of random regular graphs.
We refer the interested reader to the survey of Wormald on random regular graphs for a short introduction to the method~\cite{Wor99}.

In this paper we will consider standard switchings as well as a particular extension of them. This extension concerns pairs consisting of a  simple graph $G$ and  the  multigraph  $H_G$ obtained from $G$ by deleting its cyclic components and suppressing the vertices of degree 2 in the non-cyclic ones. 
For certain switchings of $H_G$ which yield $H'$, our extension constructs a simple graph $G'$  from $G$ such that  $H_{G'}=H'$. 
We now describe for which switchings in $H_G$ we can obtain such an $H'$ and how we do so. 

Our extension  considers directed walks (either a path or a cycle)   of   $G$  which correspond to  (oriented) edges 
 in $H_G$,
(note that  an edge of $H_G$ corresponds to exactly two such directed 
walks, even if it  is a loop and hence has only one orientation).  
We  can switch on an  ordered pair of such directed walks in $G$,
corresponding to an ordered pair of orientated distinct edges $e_1=uv$ and $e_2=xy$ of $H_G$,  
such that none of the following hold:
\begin{enumerate}[(i)]
 \item  there is an edge of $G$ between $u$ and $x$ 
which forms neither $e_1$ nor $e_2$, and the walk corresponding to $e_1$ has one edge,  
\item  there  is an edge of $G$ between $v$ and $y$ which forms neither $e_1$ nor $e_2$  and the walk  corresponding to $e_2$ has one edge,  
\item $u=x$ and the directed walk corresponding to $e_1$ has at most two edges, or
\item  $v=y$ and the directed walk corresponding to $e_2$ has at most two edges. 
\end{enumerate}

To do so, let $u=w_0,w_1,\dots,w_r=v$ be the directed walk corresponding to $e_1$ and let $x=z_0,z_1,\dots,z_s=y$ be the directed walk corresponding to $e_2$. We delete the edges $w_{r-1}v$ and $xz_{1}$ and add the edges  $w_{r-1}x$ and $vz_{1}$.

We note that (i)-(iv) ensure that we obtain a  simple graph $G'$. Furthermore, we have that $H_{G'}$ is obtained from $H_G$ by switching on $uv$ and $xy$.
We remark further that if we reverse both the ordering of the edges and the orientation of both edges, we always obtain the same graph $G'$; that is, it is equivalent to switch the ordered pair $(uv,xy)$ or the ordered pair $(yx,vu)$. 
Therefore, given two walks between $u$ and $v$ and between $x$ and $y$ (either paths or cycles) of $G$, we always consider the four following possible switches: $(uv,xy)$, $(uv,yx)$, $(vu,xy)$ and $(vu,yx)$.
We note that some of these choices might give rise to the same graph $G'$. 
However, we consider each of them as a valid switch since it will be simpler to count them considering these multiplicities.

\begin{figure}[!tbp]
\begin{center}
\includegraphics{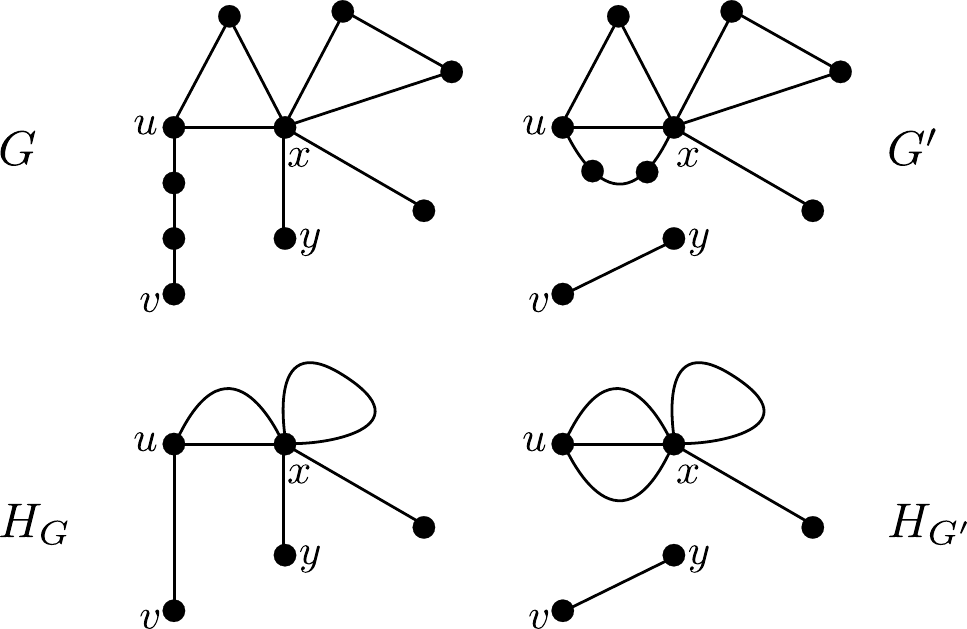}
\end{center}
\caption{\small A graph $G$ and its corresponding graph $H_G$ before and after switching the directed walks corresponding to ordered pair of oriented edges $uv$ and $xy$ in $H_G$.}\label{fig:switching}
\end{figure}

Given any two disjoint sets of (multi)graphs $\cA$ and $\cB$, we can build an auxiliary
bipartite graph with vertex set $\cA \cup \cB$ where we add an edge between $H\in \cA$ and $H'\in \cB$ for every (extended) switching that transforms $H$ into $H'$,
or equivalently, $H'$ into $H$.
We can also consider subgraphs of this auxiliary graph where we only add an edge if the switching satisfies some special property.
Given a lower bound $d_\cA$ on the degrees in $\cA$ and an upper bound $d_\cB$
on the degrees in $\cB$, we obtain immediately that $|\cA| \le \frac{d_\cB}{d_\cA}|\cB|$.
We frequently use this fact without explicitly referring to it.

To illustrate  our method, we show here that if $M_\cD$ is large with respect to the number of vertices, then there exists a component containing most of the vertices. 
\begin{lemma}\label{lem: dense case} If  $M_\cD\geq  n\log\log n$,
then the probability that $G(\cD)$ has a component of order $(1- o(1))n$ is $1-o(1)$.
\end{lemma}

In proving the lemma, we  will need the following  straightforward result on $2$-edge cuts of graphs. We defer its proof to the end of the section.
\begin{lemma}\label{lem: 2 edge cuts}
The number of pairs of orientations of edges $uv,xy$ in a graph $G$ of order $n$ such that by switching  on $uv$ and $xy$ we obtain a graph with one more component than $G$ is at most $8n^2$.
\end{lemma}
\begin{proof}[Proof of Lemma~\ref{lem: dense case}]
We can assume $n$ is  large enough to satisfy an inequality stated below since 
the lemma makes a statement about asymptotic behaviour. 
Let $K=\lfloor(1-\frac{1}{\sqrt{\log\log n}})n\rfloor$.
For every integer $k\geq 1$, let $\cFk$ be the event that $G(\cD)$ has exactly $k$ components and
let $\cFk'$ be the event that $G$ is in $\cFk$ and that all components of $G$ have order at most $K$.
Denote by $\cF'=\cup_{k\geq 2} \cFk'$.
Our goal is to show that $\bP[\cF']=o(1)$.
If so, with high probability $G$ has a component of order larger than $K$.
Observe that if one proves for some function $f:\mathbb{N}\to \mathbb{R}^+$ with $f(n)\to 0$ as $n\to \infty$ that  $\bP[\cFkk'] \le f(n)\bP[\cFk]$ for every $k\geq 1$, then $\bP[\cF']=\sum_{k\geq 1} \bP[\cFkk'] \le f(n) \left(\sum_{k\geq 1} \bP[\cFk]\right) =f(n)=o(1)$. We adopt this approach with $f(n)=\frac{16}{\sqrt{\log\log{n}}}$.

Fix $k\geq 1$. Now suppose that there exist $s^+$ and $s^-$  such that for every $G$ in $\cFk$,  there are at most $s^+$
switchings that transform $G$ into a graph in $\cFkk'$, 
and for every graph $G$ in $\cFkk'$, 
there are at least  $s^-$ switchings that transform $G$ into a graph in $\cFk$. 
Then,
\begin{align*}
	\bP[\cFkk']s^-\leq \bP[\cFk]s^+\;.
\end{align*}
Let us now obtain some values for $s^+$ and $s^-$. 
On the one hand, applying  Lemma~\ref{lem: 2 edge cuts},  we can choose
$$
s^+=8 n^2\;.
$$
On the other hand, if $G$ is in $\cFkk'$, in order to merge two components it is enough to perform a switching between an oriented non-cut edge
(at least $M_{\cD}-2n\geq (\log\log n-2)n$ choices)
and any other oriented edge not in the same component as the first one (since $G$ has minimum degree at least $1$ and the largest component has order at most $K$, there are at least $n-K$ choices).
Since $n$ is large, we can choose
$$
s^-=(\log\log{n}-2) n\cdot(n-K)\geq \frac{\sqrt{\log\log{n}}}{2}n^2\;.
$$
From the previous bounds, we obtain the desired result
\begin{align*}
	\bP[\cFkk']\leq\frac{s^+}{s^-}\cdot\bP[\cFk]\leq \frac{16}{\sqrt{\log\log{n}}}\cdot \bP[\cFk]\;.
\end{align*}
\end{proof}

\begin{proof}[Proof of Lemma~\ref{lem: 2 edge cuts}]
Any such pair of   oriented edges must lie in the same component and swapping on them within the component must yield a disconnected graph. 
So, as the function $x\mapsto 4x^2$ is convex, we may assume that $G$ is connected.

First, suppose that at least one of the oriented edges, say $uv$, is an edge cut of the graph.
Then, if $xy$ is not an edge cut of $G$, the switching does not disconnect the graph.
Since there are at most $n-1$ cut-edges, there are at most $4(n-1)^2$ switchings 
using at least one  (and hence two)  oriented cut-edges.

Otherwise $\{uv,xy\}$ is a proper 2-edge cut (that is, both $G-uv$ and $G-xy$ are connected). 
Consider an arbitrary spanning tree of $G$. 
This tree contains at least one edge of every $2$-edge cut of $G$.
Thus, select $uv$ among these edges (exactly $(n-1)$ choices). Observe now that, in order to construct the proper 2-edge cut, we need to select $xy$ as a cut-edge in $G-uv$ (at most $(n-1)$ choices).

Therefore, in total there are at most $8n^2$ switchings in $G$ which disconnect it. 
\end{proof}

\section{ The Proof of Theorem \ref{thm: main3}}\label{secmaintheorem1}

Theorem \ref{thm: main3} follows immediately from the following result.

\begin{lemma}
\label{lemexpcompsize2}
For every sufficiently small $\omega>0$ 
and every degree sequence ${\cal D}$
such that  $R_{\cal D} \leq \omega M_{\cal D}$ and $M_{\cal D}$ is sufficiently large in terms of $\omega$,
for every vertex $v$  of $H({\cal D})$, the probability that $v$
lies in a component of $H({\cal D})$ of  size larger than $\omega^{1/9}M_{\cal D}$
is less than  $e^{-M_{\cD}^{1/4}}$.
\end{lemma}
In order to prove this lemma, we analyse our random exploration process on 
$H({\cal D})$ using a set $S_0$ of vertices including $v$ and show that the
probability that there is a 
$t$ with $X_t=0$  and such that there are at most $\omega^{1/9}M_{\cal D}$ edges in the graph induced by $S_t$ in $H(\cD)$, is at least $1-e^{-M_D^{1/4}}$.

Since it is difficult to keep track of  $X_t$,
we  will instead focus on the random variable $X_t'$ (defined below), which overestimates $X_t$ until $X_t=0$.
Clearly, $X_0$ is at most
$$
X'_0=\sum_{u\in S_0}d(u)\;.
$$
Provided that $X_{t-1}>0$, the edge $v_tw_t$ is an edge of $H(\cD)$, and we can upper bound $X_t$ by
$$
X'_t= X_0'+\sum_{i=1}^t (d(w_i)-2)\;.
$$
Observe that, provided that $X_{t-1}>0$, the process $X_t'$ coincides with $X_t$ if the explored components are trees and $S_0$ is a stable set. More importantly, we observe that if $X'_t=0$, then there is a $t' \le t$ for which $X_{t'}=0$.
We note further that  the number of edges in the graph induced by $S_t$ in $H(\cD)$ is at most $X'_t+2t$, 
so we only need to show that the probability that there is no $t< \frac{\omega^{1/9}M}{2}$ 
for which $X'_t=0$ is less than $e^{-M_{\cD}^{1/4}}$.

As suggested by our introductory discussion and proven below,
the probability that $w_t$ is a specific vertex $w$ in $V(H({\cal D}))\sm {S}_{t-1}$
is essentially proportional to the degree of $w$.
Therefore, the expected value of $X'_t-X'_{t-1}=d(w_t)-2$ is with high probability close to
$$\frac{1}{\sum\limits_{w\in V(H({\cal D}))\sm {S}_{t-1}}d(w)}\sum\limits_{w\in V(H({\cal D}))\sm {S}_{t-1}}d(w)(d(w)-2).$$
By putting the high degree vertices of $H({\cal D})$ into $S_0\sm \{v\}$, we can ensure that the expectation of
$X'_1-X'_0$ is negative. 
The fact that, if the process has not died out by time $t$, 
then the sum of the degrees of the vertices we picked cannot be much less than $2t$,  
allows us to obtain a bound on the expectation of $X'_t-X'_{t-1}$ which decreases as $t$ increases.  
Having all the high degree vertices in $S_0$ 
facilitates our analysis of the exploration process, 
and allows us to show that the probability that $X'_t$ drops to $0$ 
before the number of edges in the graph induced by $S_t$ in $H(\cD)$ is at least $\omega^{1/9}M_{\cal D}$, 
is more than $1-e^{-M_{\cD}^{1/4}}$.
Forthwith the  details.

We use $H$ for $H({\cal D})$, $V$ for $V(H({\cal D}))$, $M$ for $M_{\cal D}$, $R$ for $R_{\cal D}$ and $G$ for $G({\cal D})$.  
We implicitly assume that $\omega$ is small enough and $M$ is large enough in terms of $\omega$ 
to ensure various inequalities scattered throughout the proof are satisfied.

Let $S$ be a smallest set of vertices of $H$ such that
$\sum_{i \in S} d_i  \ge 5\omega^{1/4}M$ and there
is no vertex outside of $S$ with degree bigger than a
vertex in $S$. 
Since the sum of the degrees of the vertices in $H$ is $M> 5\omega^{1/4}M$,
such a set $S$ exists.
Furthermore, since $R\leq \omega M<5\omega^{1/4}M$,
the definition of  $j_D$  implies that
$\sum_{w \in V\sm S} d(w)(d(w)-2) \leq 0$.	
It is straightforward to prove, as we do below, the following strengthening, which is important for our analysis.

\begin{lemma}
\label{lem: max degree}
We have
\begin{enumerate}[(a)]
\item $\sum\limits_{w  \in V\sm S} d(w)(d(w)-2)\leq -4\omega^{1/4} M$, and
\item there is a vertex of $S$ with degree at most  $\omega^{-1/4}$.
\end{enumerate}
\end{lemma}
The sum of the degrees in $S$ is at most the sum of $5\omega^{1/4}M$ and the minimum degree in $S$. 
Let $S_0=S \cup \{v\}$. Since every vertex not in $S$ has degree at most the minimum degree in $S$
and $X'_0$ is the sum of the degrees of the vertices in $S \cup \{v\}$, the following observation holds.

\begin{observation}\label{lem: number of edges}
We have
\begin{enumerate}[(a)]
\item $d(u)\leq \omega^{-1/4}$ for every $u\in V\setminus S_0$, and
\item $X_0'\leq 7\omega^{1/4}M$.
\end{enumerate}
\end{observation}
As we carry out our process, we let $Y_t=X'_t-X'_{t-1}-\bE[d(w_t)-2]$,
where the expectation is conditional on $\cC_t$. 
By construction $\bE[Y_t]=0$ and by Lemma \ref{lem: max degree} (b), the absolute value of $Y_t$ is bounded from above by $\omega^{-1/4}$.
As we explain
below, by applying Azuma's Inequality we immediately obtain:
\begin{lemma}
\label{lem: hahaha}
The probability that there is a $t$ such that
$\sum_{t'\leq t}  Y_{t'} >M^{2/3}$ is  less than  $e^{-M^{1/4}}$.
\end{lemma}
Thus, in order to bound $X'_t$ from above, we need to bound $\bE[d(w_t)]$ from above for each $t\geq 1$.
Letting $M_{t-1}$ be the sum of the degrees of the vertices of $H$ which are not in 
$S_{t-1}$ and using our swapping arguments, we can prove the following result, which 
will be useful to give a precise estimation for $\bE[d(w_t)]$. 
\begin{lemma}
\label{lem: prob next}
If $t \le \omega^{1/9}M$ and
$X'_{t-1}\leq \omega^{1/5}M$, and $X_{t'}>0$ for all $t'<t$, 
then the following statements hold:
\begin{enumerate}[(a)]
\item
\label{item: deg 1}
If $w\in V\sm {S}_{t-1}$ and $d(w)=1$, then
$$\Pro[w_t=w] \geq \left(1-9 \omega^{1/5}\right)\frac{1}{M_{t-1}}\;.$$
\item
\label{item: deg large}
If $w\in V\sm {S}_{t-1}$, then
$$\Pro[w_t=w] \leq \left(1+ 9 \omega^{1/5}\right)\frac{d(w)}{M_{t-1}}\;.$$
\end{enumerate}
\end{lemma}
Iteratively applying this result to bound $\bE[d(w_t)]$, we obtain:
\begin{lemma}
\label{lem: prob next2}
Letting $\tau$ be the minimum of $\lfloor \frac{\omega^{1/9}M}{2} \rfloor$ 
and the first $t$ for which $\sum_{t'\leq t} Y_{t'} > M^{2/3}$ or $X_t=0$, 
we have the following for all $t \le \tau$:
$$\bE[d(w_t)-2] \le  -\frac{t}{M}+19\omega^{1/5}.$$
\end{lemma}

The next lemma completes the proof of Lemma~\ref{lemexpcompsize2} and thus, of Theorem~\ref{thm: main3}.
\begin{lemma}
\label{lem: prob next3}
With probability greater than  $1-e^{-M^{1/4}}$, there exists $t \le \lceil \frac{\omega^{1/9}M}{3} \rceil$ such that $X_t=0$.
\end{lemma}
\begin{proof}
By Lemma \ref{lem: hahaha}, with probability  greater than $1-e^{-M^{1/4}}$ there is no $t$ such that $\sum_{t'\leq t} Y_{t'} > M^{2/3}$. 
If this event holds, and there  is no $t \le \lceil \frac{\omega^{1/9}M}{3} \rceil$
for which $X_t=0$,
then by applying Lemma~\ref{lem: prob next2} we see that for every such $t$,
$$X_t \le X'_t \le X'_0-\frac{t(t-1)}{2M} +19\omega^{1/5}t+M^{2/3}.$$
Since $X'_0 \le 7w^{1/4}M$, it follows that for $t= \lceil \frac{\omega^{1/9}M}{3} \rceil$,
we have $X_t<0$, which is a contradiction. 
\end{proof}
Therefore, it only remains to prove Lemma \ref{lem: max degree}, \ref{lem: hahaha}, \ref{lem: prob next}, and \ref{lem: prob next2}. 
 
\subsection{The Details}

We start with some simple observations.

\begin{observation}
\label{obsmaxdegree}
The maximum degree $d_{\pi_n}$ of $H$ is at most $\omega M$.
\end{observation}
\begin{proof}
By definition, $j_D\leq n$, which implies $d_{\pi_n}\leq R\leq \omega M$.
\end{proof}
We let $n_1$ be the number of vertices of degree 1 in $H$.

\begin{observation}
\label{lem: degree 1}
We have $n_1 \geq \frac{M}{3}+1$.
\end{observation}
\begin{proof}
By the definition of $j_D$, we obtain 
$$-2n_1+\sum_{i \in [j_D-1]:d_{\pi_i}\not=2}d_{\pi_i}\leq \sum_{i=1}^{j_D-1}d_{\pi_i}(d_{\pi_i}-2)\leq 0.$$
Hence,
$2n_1\geq \sum_{i\in [j_D-1]:d_{\pi_i}\not=2}d_{\pi_i}\geq M-R\geq (1-\omega)M$,
which implies
$n_1\geq \frac{(1-\omega)M}{2} \geq \frac{M}{3}+1$.
\end{proof}
If there is a vertex in $S$ of degree $1$, then every vertex outside of $S$ has degree $1$.
Thus every edge in the components of $H$ containing $S_0$ is incident to a vertex of $S_0$, 
hence there are at most $X'_0<\frac{\omega^{1/9}M}{2}$ such edges and we are done.  
So every vertex in $S$ has degree at least $3$.

\begin{proof}[Proof of Lemma \ref{lem: max degree}]
Let $j_D^*$ be such that $\sum_{i=j_D^*}^nd_{\pi_i}=\sum_{i\in S}d_i$.
Since $\sum_{i\in S}d_i>5\omega^{1/4}M>\omega M$, we have $j_D^*<j_D$.
By the definition of $R$, we have $\sum_{i=j_D^*}^{j_D-1}d_{\pi_i}\geq 5\omega^{1/4}M-R>4\omega^{1/4}M$.
Now, since $S$ contains only vertices of degree at least $3$,
the definition of $j_D$ implies
$\sum_{w\in V\setminus S}d(w)(d(w)-2)=
\sum_{i=1}^{j_D^*-1}d_{\pi_i}(d_{\pi_i}-2)\leq -\sum_{i=j_D^*}^{j_D-1}d_{\pi_i}(d_{\pi_i}-2)
\leq -\sum_{i=j_D^*}^{j_D-1}d_{\pi_i}<-4\omega^{1/4}M$,
and the first statement follows.

If every vertex in $S$ has degree at least $\omega^{-1/4}$ then, since for sufficiently small $\omega$, we have that $\omega^{-1/4}-2>\frac{3\omega^{-1/4}}{4}$.
With the above observation, we have $\sum_{i=j_D^*}^{j_D-1}d_{\pi_i}(d_{\pi_i}-2)>\frac{3\omega^{-1/4}}{4}\cdot 4\omega^{1/4}M=3M$. 
Since there are at most $M$ values of $i$ for which $d_i=1$, 
it follows that $\sum_{i=1}^{j_D-1}d_{\pi_i}(d_{\pi_i}-2)\ge 3M-M>0$, 
which is a contradiction to the choice of $j_D$,
and the second statement follows.
\end{proof}
For the proof of Lemma \ref{lem: hahaha} we recall a standard concentration inequality.
\begin{lemma}[Azuma's Inequality (see, e.g.~\cite{bookmolloyreed})]\label{lem:Azuma}
Let $X$ be a random variable determined by a sequence of $N$ random experiments $T_1,\ldots,T_N$
such that for every $1\leq i\leq N$ and any possible sequences $t_1,\ldots,t_{i-1},t_i$ and $t_1,\ldots,t_{i-1},t_i'$:
\begin{align*}
	\left|\bE[X\mid T_1=t_1,\ldots,T_{i}=t_i]-\bE[X\mid T_1=t_1,\ldots,T_{i}=t_i']\right|\leq c_i,
\end{align*}
then
\begin{align*}
	\Pro[|X-\bE[X]|>t]<2e^{- \frac{t^2}{2\sum_{i=1}^N c_i^2}}.
\end{align*}
\end{lemma}

\begin{proof}[Proof of Lemma \ref{lem: hahaha}]
Recall that $Y_t=X'_t-X'_{t-1}-\bE[d(w_t)-2]$, which implies $\bE[Y_t]=0$.
By Lemma~\ref{lem: max degree}, we have $|Y_t|\leq \omega^{-1/4}$.
Azuma's inequality applied to $\sum_{t'\leq t} Y_{t'}$ with $N=t$ and $c_i=\omega^{-1/4}$ gives us
\begin{align*}
	\Pro\left[\sum_{t'\leq t} Y_{t'}>M^{2/3}\right]<2e^{- \frac{M^{4/3}}{2\omega^{-1/2}t}}<e^{-M^{2/7}}\;,
\end{align*}
where we used that $t\leq M$. 
A union bound over all $t\leq M$ suffices to obtain the desired statement.
\end{proof}

\begin{proof}[Proof of Lemma \ref{lem: prob next}]
We can assume that there is an edge of $H$  from $S_{t-1}$ to $V\sm S_{t-1}$
as otherwise the probability that $w$ is $w_t$ is exactly $\frac{d(w)}{M_{t-1}}$, by construction of the exploration process, and we are done.

It is enough to prove this result conditioned on the current configuration $\cC_{t-1}$.
In doing so, we partition the inputs  within  $\cC_{t-1}$ (a graph with an ordering of the adjacency list of each vertex)  into different equivalence classes. All the inputs in the same equivalence class share the same underlying graph $G$, a partial order of the adjacency list of each vertex in $S_t$ and a specification of the first edge from $v_t$ to $V\sm S_{t-1}$ in the ordering
of the adjacency list for $v_t$.
Observe that each equivalence class corresponds to the same number of inputs; these arise from each other by suitably reordering some of the adjacency lists.
More precisely, every equivalence class contains
$\prod_{i\in V} (d'_i)!$ inputs, where $d'_i=d_i$ if $i \in V\sm S_{t-1}$,
and
$d'_i$ is the number of edges of $H$ from $i$ to $V\sm S_{t-1}$ if
$i \in S_{t-1}\sm \{v_t\}$ and is one less than this number if $i=v_t$.

We generalise the definition of extended switching from graphs to equivalence classes of inputs, provided that the switching neither uses nor creates an edge of $H$ within $S_{t-1}$.  This ensures that after the switching, the set of edges of $H$ within $S_{t-1} $ (and the set of edges of $G$ corresponding to them) remains unchanged. 
Consider switching the pair of edges $(e_1=uv,e_2=xy)$ as explained in Section~\ref{sec:swi}. 
Let $u=w_0,w_1,\dots,w_r=v$ be the directed walk in $G$ corresponding to $e_1$ and let $x=z_0,z_1,\dots,z_s=y$ be the directed walk in $G$ corresponding to $e_2$. When switching, we delete the edges $w_{r-1}v$ and $xz_{1}$ and add the edges  $w_{r-1}x$ and $vz_{1}$. This naturally preserves the ordering of the adjacency lists of $u=w_0,w_1,\dots, w_{r-2}$ and $z_2,\dots,z_{s-1},z_2=y$. The orderings of the adjacency lists of $w_{r-1}$ and $z_1$ are obtained by preserving the position of the edges $w_{r-2}w_{r-1}$ and $z_1z_2$ in them, respectively. If $v\neq x$, after the switching, the ordering of the adjacency list of $v$ is obtained by preserving the position of the edges incident to $v$ and different from $e_1$. If $v=x$, after the switching, the position of the edges $e_1$ and $e_2$ in the adjacency list of $v$ is swapped. This works analogously for $x$. Note that this generalisation preserves the equivalence class of the input: if a switching is performed to an input, first, no edges within $S_{t-1}$ are modified (neither their position on the corresponding adjacency lists) and, second, the specification of the first edge from $v_t$ to $V\sm S_{t-1}$ in the adjacency list of $v_t$ is preserved.

For any vertex $w\in V\sm {S}_{t-1}$,
we let ${\cal A}_w$ be the set of equivalence classes
for which $w$ is $w_t$, and let
${\cal B}_w$ be the set of equivalence classes  for which $w$ is not $w_t$.
In order to bound the probability that $w$ is chosen to be the vertex $w_t$ that will be added to $S_{t-1}$,
we consider switchings between elements of ${\cal A}_w$ and ${\cal B}_w$,
which allow us  to relate $|{\cal A}_w|$ and $|{\cal B}_w|$.

Since $t < \frac{M}{12} $, the set $V\sm S_{t-1}$ contains at least $\frac{M}{4}+1$ of the at least $\frac{M}{3}+1$ vertices of degree $1$ and so $M_{t-1}\geq\frac{M}{4}+1$.

{\it Proof of (\ref{item: deg 1}).}
A switching from ${\cal A}_w$ to ${\cal B}_w$ involves the oriented edge $v_tw$ 
and one of the remaining $M_{t-1}-1$ oriented edges with first  endpoint in $V\sm S_{t-1}$ or $wv_{t}$ and one of the remaining $M_{t-1}-1$ oriented edges with last  endpoint in $V\sm S_{t-1}$.
This implies that there are less than $4M_{t-1}|{\cal A}_w|$ switchings from ${\cal A}_w$ to ${\cal B}_w$.

Next, we prove a lower bound for the number of switchings from ${\cal B}_w$ to ${\cal A}_w$. 
The choice of an equivalence class in  ${\cal B}_w$
fixes a choice of the edge $v_tw_t$.
A vertex in $V\sm S_{t-1}$ is {\it bad} if it either has a neighbour in $S_{t-1}$,
  has a common neighbour with $w_t$ or is adjacent to $w_t$. 
If $w$ is not bad, then  for the unique neighbour $z$ of $w$  there exists four  switchings using  $v_tw_t$ and $w z$ (or their reverses), from ${\cal B}_w$ to ${\cal A}_w$.
By the hypothesis of the lemma, at most $X_{t-1}\leq X_{t-1}'\leq \omega^{1/5}M$ vertices in $V\sm {S}_{t-1}$ 
have a neighbour in $S_{t-1}$.
By Observation  \ref{lem: number of edges},
the degrees of the vertices in $V\sm {S}_{t-1}$ are at most $\omega^{-1/4}$,
which implies that there are at most $\omega^{1/5}M+\omega^{-1/2}+\omega^{-1/4} \le 2\omega^{1/5}M$ bad vertices.
Since sampling an element of ${\cal B}_w$ uniformly at random,
and applying a random permutation to the at least $M/4$ vertices of degree $1$ in $V\sm ({S}_{t-1}\cup \{ w_t\})$,
yields an equivalence class of ${\cal B}_w$ that is still sampled uniformly at random,
the proportion of equivalence classes in ${\cal B}_w$ in which a specific vertex $w$ of degree $1$ is not bad,
is at least $1-\frac{2\omega^{1/5}M}{M/4}>1-8\omega^{1/5}$.
This implies that there are at least
$4\left(1-8\omega^{1/5}\right)|{\cal B}_w|$
switchings from ${\cal B}_w$ to ${\cal A}_w$.

Double-counting these switchings yields
$|{\cal A}_w|\geq \frac{(1-8\omega^{1/5}){|\cal B}_w |}{M_{t-1}}$.
Since,
$\Pro[w=w_t] = \frac{|{\cal A}_w|}{|{\cal A}_w|+|{\cal B}_w|}$, 
and $M_{t-1}=M-X_{t-1}'-2(t-1)$ is large in terms of $\omega$, 
it follows that
$\Pro[w=w_t] \ge  \frac{1-9\omega^{1/5}}{M_{t-1}}$.

{\it Proof of (\ref{item: deg large}).}
Let $d=d(w)$.
For a switching from ${\cal B}_w$ to ${\cal A}_w$,
we have to switch the ordered edge  $v_tw_t$ with one of the $d$ ordered edges $wy$
or the ordered edge $w_tv_t$ with one of the $d$ ordered edges $yw$. 
Therefore, the number of switchings from ${\cal B}_w$ to ${\cal A}_w$ is at most $4d|{\cal B}_w|$.

Next, we prove a lower bound for the number of switchings from ${\cal A}_w$ to ${\cal B}_w$.
We have such a switching  involving the ordered  edge $v_tw$,
and some  ordered edge $yz$  provided  $y$  is in $V\sm {S}_{t-1}$, and
$v_ty$ and $wz$ are not edges. We have such a switching  involving the
 ordered  edge $wv_t$,
and some  ordered edge $zy$  provided  $y$  is in $V\sm {S}_{t-1}$, and
$v_ty$ and $wz$ are not edges.
Since $v_t$ has degree at most $\omega M$, and the maximum degree of
a vertex in $V\sm S_{t-1}$ is $\omega^{-1/4}$, the number of choices for $zy$
is at least  $M_{t-1}-\omega^{3/4} M-X'_{t-1}-\omega^{-1/2}>(1-2\omega^{1/5})M_{t-1}$.

Double-counting the number of switchings, we obtain
$|{\cal B}_w|\geq \left(1-2\omega^{1/5}\right)\frac{M_{t-1}}{d}|{\cal A}_w|$.
As before 
$\Pro[w=w_t] = \frac{|{\cal A}_w|}{|{\cal A}_w|+|{\cal B}_w|}$, and hence
$\Pro[w_t=w] \leq \frac{(1+9\omega^{1/5})d}{M_{t-1}}.$
\end{proof}

\begin{proof}[Proof of Lemma \ref{lem: prob next2}]
The proof is by induction on $t$.
Recall that there are between $\frac{M}{4}$ and $M$ vertices of degree $1$ in $V\sm S_{t-1}$.

First, let $t=1$.
Observation \ref{lem: number of edges} (b) gives us $X_0'\leq \frac{\omega^{1/5}M}{2}$,  hence  Lemma \ref{lem: prob next} implies
$$\bE[d(w_1)-2] \le  \frac{(1+9\omega^{1/5})\sum_{w\in V\sm S_{0}} d(w)(d(w)-2)+18\omega^{1/5}n_1}{M_0}.$$
Thus, by Lemma \ref{lem: max degree} and the fact that $n_1$ is at most $M_0+1$ (recall that there are no vertices of degree $1$ in $S_{0}$), we have $ \bE[d(w_1)-2] \le  -\frac{1}{M}+19\omega^{1/5}$.

Now, let $2\leq t\leq \frac{\omega^{1/9}M}{2}$.
By induction, we obtain
$$
X'_{t-1}= X'_0+ \sum_{i=1}^{t-1} (\bE[d(w_i)]-2) + \sum_{t'<t}Y_{t'}
\leq  \frac{\omega^{1/5}M}{2}+19\omega^{1/5}t+M^{2/3}
\leq  \omega^{1/5}M.$$
\begin{claim} \label{cla:0}
For any sequence of positive integers $a_1,\ldots,a_j$ distinct from 2 and a nonnegative  integer $\ell$ such that $\sum_{i=1}^j a_i\geq 2j-\ell$, we have $\sum_{i=1}^j a_i(a_i -2)\geq j-2\ell$.
\end{claim}
\begin{proof}
The proof is by induction on $\sum_{i=1}^j a_i$.
The statement is trivially true if $j=0$.

If all $a_i$ are at least $3$, then 
$\sum_{i=1}^j a_i(a_i-2)\geq \sum_{i=1}^j a_i\geq 2j-\ell\geq j-2\ell$.
Hence, we may assume without loss of generality that $a_j=1$. 
Since
$\sum_{i=1}^{j-1}a_i \ge 2j-\ell-1= 2(j-1)-(\ell-1)$,
if $\ell \neq 0$, we obtain by induction that 
$\sum_{i=1}^j a_i(a_i -2)\geq (j-1)-2(\ell-1)-1=j-2\ell$.

So, we can assume that $\ell=0$. 
If for all $i\in [j]$, the integer $a_i\in \{1,3\}$,
then the claim also follows easily, as then $|\{i\in [j]:a_i=3\}|\geq |\{i\in [j]:a_i=1\}|$.
Since $3(3-1)+1(1-2)=2$,
the negative contribution of a $1$ is compensated by a $3$.

Hence, without loss of generality, $a_1\geq 4$ and $j>1$.
If  $\sum_{i=1}^j a_i > 2j$ then  by decreasing $a_1$ by 1 
and then applying induction we are done. 
So, we may assume $\sum_{i=1}^j a_i = 2j$, and thus there must be at least $a_1-2$ values of $i$ for which $a_i=1$.  The sum of $a_k(a_k-2)$ 
over these $a_1-2$ elements and $a_1$ is $(a_1-1)(a_1-2) \ge a_1-1$.
So, deleting these $a_1-1$ 
elements from the set and again applying induction, we are done. 
%\guillem{I don't understand the sentence ``or the fact that the sequence is 2''}
\end{proof}

Since $X_{t-1}>0$, we have $X'_{t-1}>0$, and hence
$\sum_{i=1}^{t-1} d(w_i)=2(t-1)+\sum_{i=1}^{t-1} (d(w_i)-2)=2(t-1)+X'_{t-1}-X'_0 \geq 2(t-1)-X'_0$. 
By Claim~\ref{cla:0}, $$\sum_{i=1}^{t-1} d(w_i)(d(w_i)-2)\geq (t-1)-2X'_0.$$
Since $V\sm S_{t-1}=V\sm (S_0\cup\{w_1,...,w_{t-1}\})$, Lemma~\ref{lem: max degree} yields 
$$
\sum_{w\in V\sm {S}_{t-1}} d(w)(d(w)-2) \leq  2X'_0-(t-1)\;.
$$
Combining this bound with Lemma \ref{lem: prob next} yields to
$$\bE[d(w_t)]-2
\le \frac{(1+9\omega^{1/5})(2X'_0-(t-1))+18\omega^{1/5}n_1}{M_{t-1}}\;.$$
Since $n_1 \le M_{t-1}$, $X'_0 \le 7\omega^{1/4}M$, and $\frac{M}{4} \le M_{t-1}<M$, we obtain
$$\bE[d(w_t)]-2
\le -\frac{t-1}{M}+ 4(1+9\omega^{1/5})(14\omega^{1/4})+18\omega^{1/5}
\le -\frac{t}{M}+19\omega^{1/5}$$
which completes the proof of Lemma \ref{lem: prob next2}.
\end{proof}

\section{ The Proof of Theorem \ref{thm: main4}}\label{secmaintheorem2}

Again, we use $H$ for $H({\cal D})$, $V$ for $V(H({\cal D}))$, $M$ for $M_{\cal D}$, $R$ for $R_{\cal D}$ and $G$ for $G({\cal D})$.  Since ${\cal D}$ is well-behaved, we can also assume that $M$ is sufficiently large to satisfy various inequalities scattered throughout the proof. 

As we have already mentioned, in the preprocessing step we need to handle the vertices of
large degree. We let $L$ be the set of vertices of degree exceeding $\frac{\sqrt{M}}{\log M}$.
We will prove the following result using the combinatorial switching approach:

\begin{lemma}
\label{preprocesslem1}
The  probability that the number of edges in  the components of $H$ intersecting  $L$
is at least $\frac{R}{200}$ and there exists no component containing at least
$\frac{R}{200}$ edges is $o(1)$. 
\end{lemma}

We also need the following straightforward result:

\begin{lemma}\label{lem:imp}
Let $U$ be a set of vertices of $H$ that contains all vertices of degree exceeding $\frac{\sqrt{M}}{\log M}$
and let $\frac 14 <c<1$ be such that $\sum_{u\in U}d(u) \leq c R$.
Then  $$\sum_{u\in V\sm U}d(u)(d(u)-2) \geq  \frac{(1-c)}{2}R.$$
\end{lemma}
\begin{proof}
The assumptions imply $d_{\pi_{j_\cD}}\leq \frac{\sqrt{M}}{\log M}$.
Letting $U'=\{i:\, i\leq  j_\cD,\, \pi_i\in U\}$ and $K=\sum_{i\in U'} d_{\pi_i}$,
the definition of $j_\cD$ implies
\begin{align*}
\sum_{u\in V\sm	U}d(u)(d(u)-2)
&\geq  - \sum_{i\in U'} d_{\pi_i}(d_{\pi_i}-2)+ \sum_{i=j_\cD+1}^n d_{\pi_i}(d_{\pi_i}-2)- \sum_{i\in U\sm U'} d_{\pi_i}(d_{\pi_i}-2) \\
&\geq -K(d_{\pi_{j_\cD}}-2)+(d_{\pi_{j_\cD}}-2)(R-d_{\pi_{j_\cD}}-(cR-K))\\
&\geq \frac{(1-c)}{2}R.
\end{align*}
\end{proof}
Starting with the set $L$, we first explore all the components in $H$ that contain at least one vertex in $L$.
Let $U$ be the set of all vertices in such components.
If $\sum_{u\in U} d(u)\geq \frac{R}{100}$, then by Lemma~\ref{preprocesslem1}
the probability that  there  does not exist  a component in $H$ with at least $\frac{R}{200}$ edges is $o(1)$.  
So, in what follows, we condition on $\sum_{u\in U} d(u)\leq \frac{R}{100}$.
 
We let $S_0=U\cup\{v\}$, where $v$ is a random vertex selected with probability proportional to $d(v)$.
Since there are no edges from $U$ to $V\sm U$, we have $v_1=v$. Note that this implies that, for every $t$, the edges counted by $X_t$ belong to the same component (not necessarily the one of $v$). 
In addition, the maximum degree of a vertex in $V\sm U$ is at most $\frac{\sqrt{M}}{\log{M}}$
and $M\geq M_0\geq M-\frac{R}{100}-d(v_0)\geq M-\frac{M}{100}-\frac{\sqrt{M}}{\log{M}}\geq \frac{98M}{100}$.

For each vertex $w\in V\sm S_{t-1}$, we let $d'_t(w)$ be the sum of  the number of loops of $H$ at $w$ and the number of
edges of $H$  between $w$ and $S_{t-1}\sm\{v_t\}$. 
Observe that we can control the number of edges between $S_t$ and $V\sm S_t$ as follows
\begin{align}\label{eq:edges}
X_t=X_{t-1}+(d(w_t)-2) - 2d'_t(w_t)\;.
\end{align}
The next lemma shows that the probability of selecting a vertex $w$ at time $t$ is essentially proportional to its degree.
\begin{lemma}\label{lem: w bad}
Let $\beta<10^{-6}$ be a fixed constant.
If $M_{t-1}\geq \frac{3M}{4}$ and $X_{t-1}\leq \beta M$,
then for every $w\in V\sm S_{t-1}$,
$$
(1-10\sqrt{\beta})\frac{d(w)}{M_{t-1}} \le \Pro[w=w_t]\leq  (1+10\sqrt{\beta})\frac{d(w)}{M_{t-1}}\;,
$$
and,
$$
\Pro\Big[d'_t(w) \ge \lfloor 2\sqrt{\beta}d(w) \rfloor +i|w=w_t\Big]  \leq   \beta^{i/2}\;.
$$
\end{lemma}
We are now in a position to carry out our exploration process.
Let $A_t=d(w_t)-\bE[d(w_t)]$ and let $B_t=d'_{t}(w_t)-\bE[d'_{t}(w_t)]$. By our convention, both expectations are conditional on $\cC_{t-1}$.
We let $\cF_{bad}$ be the set of those inputs such that for some $t$ either
$\sum_{t'\leq t} A_{t'} >\frac{M}{\log \log M}$ or $\sum_{t'\leq t} B_{t'} >\frac{M}{\log \log M}$.
\begin{lemma}
\label{lem: hahaha1}
$\Pro[\cF_{bad}]=o(1)$.
\end{lemma}
Let $\epsilon' =\frac{R}{M}\geq \epsilon$, $ \beta=10^{-6} \epsilon^2$, and
$\tau$ be the smallest $t$ for which either $X_t \ge \beta M$ or $M_t \le \left(1-\frac{\epsilon'}{4}\right)M_0$.

\begin{lemma}
\label{thisisit}
For any $t<\tau$,  $\bE[d(w_t)-2]\geq \frac{\epsilon}{4}$, $\bE[d'_{t}(w_t)] \le \frac{\bE(d(w_t)-2)}{3}$, 
and $\bE[X_t-X_{t-1}]\geq \frac{\bE[d(w_t)-2]}{3}\geq \frac{\epsilon}{12}$.
\end{lemma}

\begin{proof}
We have
$$\sum_{u\in S_{t-1}} d(u) = \sum_{u\in S_0} d(u)+ \sum_{u\in S_{t-1}\sm S_0} d(u)\leq \frac{R}{100}+M_0-M_{t-1}\leq \left(\frac{1}{100}+\frac{1}{4}\right)R\leq \frac{R}{3}\;.$$
By the definition of $\tau$, the hypotheses of Lemma~\ref{lem: w bad} are satisfied.
Using Lemma~\ref{lem:imp} as well as Lemma~\ref{lem: w bad} we conclude
\begin{align*}
\bE[d(w_t)-2]&= \sum_{w\in V\sm S_{t-1}} (d(w)-2) \Pro[w=w_t]\\
&\geq \frac{1}{M_{t-1}}(1-10\sqrt{\beta})\!\!\!\! \sum_{\genfrac{}{}{0pt}{}{w\in V\sm S_{t-1}}{d(w)\geq 3}}\!\!\!\!\!\! d(w)(d(w)-2)+\frac{1}{M_{t-1}}(1+10\sqrt{\beta})\!\!\!\! \sum_{\genfrac{}{}{0pt}{}{w\in V\sm S_{t-1}}{d(w)= 1}}\!\!\!\!\!\!d(w)(d(w)-2)\\
&\geq \frac{1}{M_{t-1}}(1-10\sqrt{\beta}) \sum_{w\in V\sm S_{t-1}} d(w)(d(w)-2)- \frac{20\sqrt{\beta}n_1}{M_{t-1}}\\
&\geq \frac{1}{3}\left(1-10\sqrt{\beta}\right)\frac{R}{M}-30\sqrt{\beta} \geq \frac{\epsilon}{4}\;,
\end{align*}
since $\beta\leq 10^{-6}\epsilon^2$. This proves the first statement.
Again, by Lemma~\ref{lem: w bad}, we obtain
$$
\bE[d'_{t}(w_t)]
\le 7\sqrt{\beta}\bE[d(w_t)]
=   7\sqrt{\beta}\bE[d(w_t)-2]+14\sqrt{\beta}
\le \frac{\bE[d(w_t)-2]}{6}+\frac{\epsilon}{70}
\le \frac{\bE[d(w_t)-2]}{3}\;,
$$
where the last inequality follows from the first statement of this lemma.

Now, since $\bE[X_t-X_{t-1}]=\bE[d(w_t)-2]-\bE[2d'_{t}(w_t)]$ the third statement follows directly from the first and second one.
\end{proof}
Since all the edges counted by $X_t$ are in the same component of $H$, this next lemma proves Theorem~\ref{thm: main4}.
\begin{lemma}
With probability $1-o(1)$, we have $X_\tau \ge \beta M$.
\end{lemma}
\begin{proof}
We show that if our configuration is not in ${\cal F}_{bad}$ then $X_\tau \ge \beta M$.
By Lemma~\ref{lem: hahaha1}, the result follows. 

Applying \eqref{eq:edges} recursively, we have
\begin{align*}
X_\tau &= X_0+\sum_{t\leq\tau}(d(w_t)-2) -2\sum_{t\leq\tau}d'_{t}(w_t)\;.
\end{align*}
By adding $\bE[X_\tau]$, subtracting the expectation of the right hand side in the previous equation and since $\bE[X_0]=X_0$, we obtain that
\begin{align*}
X_\tau &= \bE[X_\tau] + \sum_{t\leq\tau}(d(w_t)-2-\bE[d(w_t)-2])-  2  \sum_{t\leq\tau} (d'_{t}(w_t)-\bE[d'_{t}(w_t)])\\
&=     \bE[X_\tau]+ \sum_{t\leq\tau}A_t -2\sum_{t\leq\tau}B_t\\
&\geq      \bE[X_\tau]- \frac{3M}{\log\log M}.
\end{align*}
If $\tau>\lceil\frac{\epsilon' M}{400}\rceil$,
then Lemma~\ref{thisisit} implies
$X_{\tau}\geq \frac{\epsilon \tau}{12}-\frac{3M}{\log\log M}\geq \beta M$, 
and we are done.
Now, let $\tau\leq \lceil\frac{\epsilon' M}{400}\rceil$.
If $X_{\tau} < \beta M$, then, by the definition of $\tau$, 
$M_{\tau} \le \left(1-\frac{\epsilon'}{4}\right)M_0$. 
Note that $\sum_{t\leq \tau} d(w_{t})= M_0-M_{\tau}\geq \frac{\epsilon'}{4}M_0\geq \frac{\epsilon'}{5}M$,
because $M_0\geq \frac{98M}{100}$. 
%Arguing as above 
Using Lemma~\ref{thisisit} as before, we obtain
\begin{align*}
X_{\tau}
&\geq X_0+\sum_{t\leq\tau} \frac{\bE[d(w_t)-2]}{3} - \frac{3M}{\log\log M}
 \geq \frac{\epsilon'}{15}M  -\frac{2\tau}{3}- \frac{3M}{\log\log M}
 \geq \frac{\epsilon'}{20}M
 >\beta M,
\end{align*}
a contradiction. 
Thus, $X_\tau > \beta M$ in both cases.
\end{proof}

It remains to prove Lemma \ref{preprocesslem1},~\ref{lem: w bad} and~\ref{lem: hahaha1}.

\subsection{The Details}

We start this section with a result showing that if there are many vertices in $L$, then they all lie in the same component of $H$.
\begin{lemma}\label{lem: large deg one comp}
If  $|L| \ge \log^7 M$, then the probability the vertices of $L$ lie in the same component of $H$ is $1-o(1)$.
\end{lemma}
\begin{proof}
Let $L_6$ and $L_7$ be the vertices of degree at least $\frac{\sqrt{M}}{\log^6{M}}$ and $\frac{\sqrt{M}}{\log^7{M}}$, respectively.
We divide the proof into two cases depending on the size of $L_6$.

\bigskip

\noindent
\textbf{Case 1:} $|L_6|\geq\frac{\sqrt{M}}{\log^6{M}}$.

\medskip

We begin with a claim which shows that every vertex in $L$ is adjacent in $H$ to a large number of vertices in $L_7$.

\begin{claim}\label{cla:1}
For every $u\in L$, the probability that $u$ is adjacent to at most $\frac{\sqrt{M}}{\log^{14}{M}}$ vertices in $L_7$ is at most $M^{-7}$.
\end{claim}
\begin{proof}
Let $K=\lceil\frac{2\sqrt{M}}{\log^{14}{M}}\rceil$.
Assume for a contradiction that the claim
fails for $u\in L$.
For every $k\in \{0,\dots,K\}$,
let $\cFk$ be the event that $u$ is adjacent to exactly $k$ vertices in $L_7$.
By our assumption,
there is some $k_0\in \{0,\dots,\frac{K}{2}\}$ such that
$\Pro[{\mathcal F}_{k_0}]>M^{-8}$.

Suppose that $G$ is in  $\cFk$. We consider switchings that lead to a multigraph in either $\cFkk$ or $\cFkkk$. We stress here that we will use a specially adjusted version of switchings.
Consider edges $uv\in E(H)$ such that $v \notin L_7$
or $uv$ is not an edge in $G$. We have at least $\frac{\sqrt{M}}{\log M}-k\geq \frac{\sqrt{M}}{2\log M}$ choices for such an edge.
Moreover, there are at least  $\frac{\sqrt{M}}{\log^6 M}-\frac{\sqrt{M}}{\log^{14} M}\geq \frac{\sqrt{M}}{2\log^6 M}$ vertices $x\in L_6 \sm N_H(v)$. Now we discuss different switching situations depending on the structure of $G$. 

First, suppose that there are at least $\frac{\sqrt{M}}{4\log{M}}$ edges $uv$ such that $v\notin L_7$. Choose such an oriented edge $uv$. Then, for each $x \in L_6 \sm N_H(v)$, there are at least $\frac{\sqrt{M}}{\log^6{M}}- \frac{\sqrt{M}}{\log^7{M}}\geq \frac{\sqrt{M}}{2\log^6{M}}$ edges $xy$ such that $y\neq v$ and either $xy$ is not an edge in $G$ or $vy\notin E(H)$. In both cases we get at least $\frac{M^{3/2}}{16\log^{13} M}$ switchings which increase the degree of $u$ in $L_7$. 

Otherwise, there are at least $\frac{\sqrt{M}}{4\log{M}}$ edges $uv$ that are not edges in $G$. Choose such an edge $uv$. Next, suppose that there are at least $\frac{\sqrt{M}}{4\log^6 M}$ vertices $x\in L_6\sm N_H(v)$ such that there are at least $\frac{\sqrt{M}}{2\log^6{M}}$ edges $xy$ with $y\neq v$. As before this give rise to at least $\frac{M^{3/2}}{32\log^{13} M}$ switchings. (Observe that if $u=v$, then the obtained graph is in $\cFkkk$. Observe also that we chose to switch so that  the new edge between $y$ and $v$
corresponds to the edge between $u$ and $v$ and hence has an internal vertex). 
Otherwise, there are at least $\frac{\sqrt{M}}{4\log^6 M}$ vertices $x\in L_6\sm N_H(v)$ such that there are at least $\frac{\sqrt{M}}{2\log^6{M}}$ edges $xy$ with $y=v$. Choose such an $xy$ that it is not an edge in $G$ (all but at most one of them are not edges of $G$). If either $uv$ or $xy$ corresponds to a path of length at least $3$ in $G$, then there exists at least one switching (the one that switches such an edge to a new loop in $y=v$) that transforms $G$ into a graph in $\cFkk$. If both $uv$ and $xy$ correspond to paths of length $2$, then we perform a special type of switching. Let $uwv$ and $xzy$ be the corresponding paths in $G$. Then, we obtain the switched graph by deleting the edges $uw$, $wv$, $xz$ and $zy$ and by adding the edges $ux$, $vw$, $wz$ and $zy$. This gives a graph in $\cFkk$. In this case, there are also at least $\frac{M^{3/2}}{32\log^{13} M}$  switchings.

Now, for any $G$ in either $\cFkk$ or $\cFkkk$, consider the switchings that transform it into a multigraph in $\cFk$. 
We must use an edge $uv$ for $v\in L_7$ which is not a parallel edge in $H$. While there might be many edges between $u$ and $L_7$, note that there are at most $k+2$ of this type. We can select $xy$ in at most $M$ ways.
Thus there are at most $4(k+2)M\leq \frac{10 M^{3/2}}{\log^{14} M}$
switchings leading to a multigraph in $\cFk$. The factor $4$ comes from the fact that we performed the special type of switching introduced above, that given two edges in a graph can give rise to at most $4$ graphs.

Hence $\Pro[\cFkk] +\Pro[\cFkkk] \geq \frac{\log{M}}{320} \bP[\cFk] \geq 8\bP[\cFk]$, and in particular $\max\{\Pro[\cFkk],\Pro[\cFkkk]\} \geq 4\bP[\cFk]$. Using that $\Pro[{\mathcal F}_{k_0}]\geq M^{-8}$, we have $\max\{\Pro[{\mathcal F}_{K-1}], \Pro[{\mathcal F}_K]\}\geq 2^{K-k_0}\Pro[\cF_{k_0}] >1$, which is a contradiction.
\end{proof}

Now we use Claim~\ref{cla:1} to show that any two vertices in $L$ whose degree is not extremely large, lie in the same component.
\begin{claim}\label{cla:2}
For every $u,v\in L$  each of degree at most $\frac{M}{\log^{24} M}$, the probability that they are not in the same component is at  most $M^{-4}$.
\end{claim}
\begin{proof}
Let $K=\lceil\log{M}\rceil$. For every $k\in \{0,\dots,K\}$, we define the following events,
\begin{itemize}
\item[] $A_1:$ $u$ and $v$ have no common neighbour in $H$.
\item[] $A^k_2:$ there  are $k$ edges  between $N_H(u)$ and $N_H(v)$ in $H$.
\item[] $A^k_3:$ $H$ has an edge-cut of size at most $2k$ separating $N_H(u)$ and $N_H(v)$.
\end{itemize}
Let $\cFk$ be the event $A_1\cap A^k_2\cap A^k_3$. 
Observe that if $\cF_0$ is not satisfied, then there exists a path between $u$ and $v$. Thus, it suffices to show $\Pro[\cF_0] \leq M^{-4}$.

Here, we will show that for every $k$ satisfying $\Pro[\cFk]\geq M^{-4}$, we have 
$$
\max\{\Pro[\cFkk], \Pro[\cFkkk]\}\geq \log M\cdot \Pro[\cFk]\;.
$$
This implies that $\Pro[\cF_0]\leq\max\{M^{-4},(\log{M})^{-\frac{K}{2}}\}\leq M^{-4}$ and proves the claim.

So, suppose $\Pro[\cFk]\geq M^{-4}$.
Let $\cFk'\subseteq \cFk$ be the event that $u,v$ are in addition adjacent to at least $\frac{\sqrt{M}}{\log^{14}{M}}$ vertices in $L_7$. By Claim~\ref{cla:1} and since $\Pro[\cFk]\geq M^{-4}$, we obtain that  $\Pro[\cFk']\geq \frac{1}{2}\Pro[\cFk]$.

We consider switchings from a graph $G$ in $\cFk'$ to $\cFkk$ or $\cFkkk$,
which use no edge incident to $u$ or $v$.
We are going to switch using edges from two sets which we now define.  

Fix an edge cut $F_1$ of size at most $2k$ separating $N_H(u)$ and $N_H(v)$ and let $F_2$ be the set of $k$ edges between $N_H(u)$ and $N_H(v)$. These two sets of edges exist by $A^k_3$ and $A^k_2$,  respectively.

Given that $G$ is in $\cFk'$, there are at least $\frac{\sqrt{M}}{\log^{14}{M}}$ vertices $x\in N_H(u)\cap L_7$ and for each such $x$, there are at least $\frac{\sqrt{M}}{\log^{7}{M}}$ edges $xy$. Since $d(u)\leq \frac{M}{\log^{24} M}$, essentially all such $xy$ satisfy $y\neq u$. Indeed, we can find a set $X$ of $\frac{\sqrt{M}}{2\log^{14}{M}}$ vertices $x\in N_H(u)\cap L_7$ such that $x$ is not an endpoint of $F_1\cup F_2$ and there are at least $\frac{\sqrt{M}}{2\log^{7}{M}}$ edges $xy$ with $y\neq u$. Let $E_1$ be the set of edges $xy$ such that $x\in X$, $y\neq u$ and either $xy$ is not an edge of $G$ or $y$ is not an endpoint of any edge of $F_1 \cup F_2$. Since $|F_1\cup F_2|\leq 3k$, $|E_1|\geq \frac{M}{8\log^{21}{M}} $.

In the same vein, we can obtain a set of vertices $W$ 
such that for each $w\in W$, we have $w\in N_H(v)\cap L_7$, $w$ is not an endpoint of $F_1\cup F_2$ and there are at least $\frac{\sqrt{M}}{2\log^{7}{M}}$ edges $wz$ with $z\neq v$. Moreover, we can also obtain a set of edges $E_2$ with $|E_2|\geq \frac{M}{8\log^{21} M}$ such that for each $wz\in E_2$, $w\in W$, $z\neq v$ and either $wz$ is not an edge of $G$ or $z$ is not an endpoint of any edge of $F_1 \cup F_2$.

Observe that for any $xy\in E_1$ and any $wz\in E_2$, we have $y\neq z$. Otherwise, if $y=z$, there exists a path $uxywv$ non of whose edges are in the edge cut $F_1$, getting a contradiction. 

If $yz$ is an edge of $H$, then $yz\in F_1$ and both $xy$ and $wz$ are not edges of $G$. Note that $xw\notin E(H)$. Thus, we can always switch $xy$ and $wz$ to obtain a new graph which is $\cFkk$ or $\cFkkk$ (it only belong to $\cFkkk$ if $y\in N_H(u)$ and $z\in N_H(v)$). There are  at least 
$\frac{M^2}{64\log^{42} M}$ switchings.  

Given a graph in $\cFkk \cup \cFkkk$, there are at most $4(k+2)M\leq M^{3/2}$ switchings which yield a graph in $\cFk$.

We conclude that $\max\{\Pro[\cFkk], \Pro[\cFkkk]\}\geq \log M\cdot \Pro[\cFk]$, as desired.
\end{proof}

A union bound over all pairs $u,v\in L$ together with Claim~\ref{cla:2} suffices to show that in Case~1, with probability $1-o(1)$ all the vertices in $L$ with degree at most  $\frac{M}{\log^{24}M}$  lie in the same component of $H$.

Now we consider a set $S$ consisting of all the vertices of $L$ of degree at least $\frac{M}{\log^{24}M}$
and one other vertex of $L$, if there are any more. Therefore, $|S|\leq\log^{24} M+1$.

For any $u,v\in S$, we let $\cA_{u,v}$ be the event that $u$ and $v$ are in  the same component and $\cB_{u,v}$ be the event that they are in different components. We will use switchings involving $u$ and $v$.
On the one hand, for any graph in $\cB_{u,v}$, there are $d(u)d(v)\geq \frac{M^{3/2}}{\log^{25} M}$ switchings which yield a graph $G$  in
$\cA_{u,v}$ in which $uv$ is an edge of $H$.
On the other hand, for any graph $G$ in $\cA_{u,v}$ there are at most $4M$ switchings using the edge $uv$ and another edge, that transform the graph into a graph in $\cB_{u,v}$.
Therefore, $\Pro[\cB_{u,v}]\leq M^{-1/3}$. So, with probability $1-o(1)$ all the vertices in $S$ lie in the same component. This implies that with probability $1-o(1)$, all the vertices in $L$ lie in the same component of $H$ and this completes the proof of the lemma for Case 1.

\bigskip

\noindent
\textbf{Case 2:} $|L_6|\leq \frac{\sqrt{M}}{\log^6{M}}$.

\medskip
Let $\ell$ be the size of $L$.
By the hypothesis of the lemma and of this case, we have
\begin{align}\label{eq:bound}
\log^7 M\leq \ell \leq \frac{\sqrt{M}}{\log ^6M}\;.
\end{align}

The following claim shows that with high probability,  the multigraph induced in $H$ by the vertices in $L$, has large minimum degree.
\begin{claim}\label{cla:3}
With probability at least $1-\ell^{-7}$, every vertex $u\in L$ is incident to at least $\sqrt{\ell\log{\ell}}$ edges of $H$ which join it to other vertices of $L$.
\end{claim}

\begin{proof}
Fix a vertex $u\in L$.
For every $0\leq k < 2\sqrt{\ell\log{\ell}}$,
let $\cFk$ be the event that $u$ is incident to exactly $k$ edges joining it to vertices of $L$ in $H$.
Using~\eqref{eq:bound}, we have that $k\leq 2\sqrt{\ell \log{\ell}}\leq M^{1/4}$.

Let $G$ be a  graph in $\cFk$. We will count how many (extended) switchings lead to a graph in $\cFkk$.
By the hypothesis of Case 2,
there are at least $\frac{\sqrt{M}}{\log M}-\frac{\sqrt{M}}{\log^6 M}-k\geq \frac{\sqrt{M}}{2\log M}$ edges $uv$
such that  either $v \not \in L_6$ or  $v\not\in L$ and $uv$ corresponds to a path of length at least $2$ in $G$.

For any such  edge $uv$, we can switch with any edge $xy$ disjoint from $uv$
such that $x\in L\sm\{u\}$ and is not adjacent to $u$
unless one of the following situation happens:
\begin{enumerate}[(i)]
	\item $xy$ and $uv$ both correspond to edges of $G$ and there is an edge corresponding to an edge of $G$ between $y$ and $v$, or
	\item  $v=y$.
\end{enumerate}
There are at least $\ell-k-1\geq \frac{\ell}{2}$ choices for $x \in L\sm N_H(v)$.
Given the choice of $uv$ and $x$,  since $x \in L$  and if $uv$ is an edge of $G$ then $v \in L_6$, there are at least  $\frac{\sqrt{M}}{2 \log{M}}$ choices for an edge $xy$ that  do not satisfy $(i)$. Since $v \not\in L$, there are in total at most $\frac{\sqrt{M}}{\log M}$ edges $xy$ satisfying (ii) for $uv$.

Thus, there are at least $\frac{\ell\sqrt{M}}{4\log M}-\frac{\sqrt{M}}{\log{M}} \geq \frac{\ell\sqrt{M}}{5\log M}$ choices for an edge $xy$ that give a valid switching with $uv$. So, in total there are at least $\frac{\ell M}{10 \log^2 M}$ switchings. 

If $G$ is in $\cFkk$, then there are at most $4(k+1)M \le 8 \sqrt{\ell \log \ell}M$ switchings that lead to a multigraph in $\cFk$.

Using~\eqref{eq:bound}, we conclude,
\begin{align*}
\bP[\cFk] 	\leq \frac{80\sqrt{\ell \log{\ell}} \log^{2}M}{\ell} \cdot\bP[\cFkk] \leq \frac{1}{\log M} \cdot\bP[\cFkk],
\end{align*}
In particular, for every $k\leq \sqrt{\ell\log{\ell}}$, we have
\begin{align*}
	\bP[\cFk] \leq (\log{M})^{-\sqrt{\ell\log{\ell}}}\cdot \bP[\cF_{2\sqrt{\ell\log{\ell}}}]<\ell^{-9}.
\end{align*}
A union bound for all $k\leq \sqrt{\ell\log{\ell}}$ and all vertices $v\in L$ now yields to the
desired result.
\end{proof}

To complete Case $2$,
we will use the minimum degree within the vertices in $L$ to show that the multigraph induced by $L$ in $H$, denoted by $H[L]$, is connected.
For every $k\geq 1$, let $\cFk$ be the event that $H[L]$ has exactly $k$ components.
We show that  there is an $f(l)$ which is $o(l)$ such that for every $k\geq 1$ that satisfies $\bP[\cFkk]\geq \ell^{-2}$, we have $\bP[\cFkk]  \le f(l)\bP[\cFk]$.
If so, $\bP[\cF_{1}]=1-o(1)$, or in other words, with probability $1-o(1)$ the multigraph $H[L]$ is connected. This proof follows the same lines as the one in Lemma~\ref{lem: dense case}.

Fix $k\geq 1$ such that $\bP[\cFkk]\geq \ell^{-2}$.
Suppose $G$ is in  $\cFk$.
Any (extended) switching from $G$ that leads to a graph in $\cFkk$
creates a new component
and hence either uses two cut edges or uses a 2-edge cut which does not
contain a cut edge.
By Lemma~\ref{lem: 2 edge cuts},
there are at most $8\ell^2$ switchings leading to a multigraph in $\cFkk$.

For every $k\geq 1$, let $\cFk'$ be the event $\cFk$ with the additional restriction that $H[L]$ has minimum degree at least $\sqrt{\ell \log{\ell}}$.
Since  $\Pro[\cFkk]\geq \ell^{-2}$, by Claim~\ref{cla:3},  $\Pro[\cFk']\geq \frac{1}{2}\Pro[\cFk]$.
Suppose now that $G$ is in  $\cFkk'$.
We will lower bound the number of  (extended) switchings to graphs in $\cFk$.
In order to merge two components it is enough to select
non-cut edge $xy$ and an edge $uv$ in another component.
By the definition of $\cFkk'$,
there are at least $|E(H[L])|-\ell\geq \frac{1}{2}\ell^{3/2}\sqrt{\log{\ell}} - \ell\geq \ell^{3/2}$ choices for $xy$. Given the choice of $xy$, there is at least one vertex in another component, and hence, there are at least $\sqrt{\ell\log{\ell}}$ choices for $uv$.
The total number of switchings is at least $\ell^{2} \sqrt{\log{\ell}}$.

Hence, for every $k\geq 1$ and since $\ell\to\infty$ as $n\to\infty$,
$$
\bP[\cFkk]\leq 2\bP[\cFkk'] \leq 2\cdot 8\ell^2 \cdot \frac{1}{\ell^2\sqrt{\log{\ell}}}\cdot \bP[\cFk] = \frac{16}{\sqrt{\log{\ell}}}\bP[\cFk]).
$$
This completes the proof of Lemma~\ref{lem: large deg one comp}.
\end{proof}

We proceed with the proof of Lemma~\ref{preprocesslem1}.
\begin{proof} [Proof of  Lemma \ref{preprocesslem1}]

If $|L|\geq \log^7{M}$, we can use Lemma~\ref{lem: large deg one comp} to show that with probability $1-o(1)$, all the vertices in $L$ lie in the same connected component, and hence the statement of the lemma holds in this case.

Therefore, we can assume that $|L|\leq \log^7 M$. Since $R\geq \epsilon M$, this implies that if the union of the components intersecting $L$ have at least $\frac{R}{200}$ edges there is a component of size at least $M^{2/3}$ which contains a vertex of $L$. So, it is enough to  prove that for every pair $u,v\in L$,
the probability that $u$ is in a component with at least $M^{\frac{2}{3}}$ edges not containing $v$ is $o(M^{-\frac{1}{10}})$.

Fix $u,v\in L$.
Let $\cF_-$ be the event that the component of $u$ in $H$ has at least $M^{\frac{2}{3}}$ edges and $u,v$ are in different components.
Let $\cF_+$ be the event that $u,v$ are in the same component of $H$.
We will show that $\Pro[\cF_-]=o(M^{-\frac{1}{10}})\Pro[\cF_+]$.

Let $G$ be a graph in $\cF_-$. Since $v\in L$, there are at least $\frac{\sqrt{M}}{2\log{M}}$ oriented edges $vw$ (i.e. we count loops twice) and at least $M^{2/3}$ oriented edges $xy$ in the same component as $u$  ordered in such a way that $x$ is at least as close to $u$, as $y$. Thus, the total number of switchings, using an edge $xy$ ordered in such a way  leading to a multigraph in $\cF_+$ with $vx$ an edge is at least
$\frac{M^{7/6}}{\log M}$.

Consider  $G$ in  $\cF_+$ obtained by such a swap.
If $v\neq w$ and $x\neq y$, then there exists a unique edge $vx$ in any shortest path from $u$ to $v$. Otherwise we can find two edges incident to $v$, such that every shortest path from $u$ to $v$ contains one of them.
So, the total number of such switchings leading to $G$  is at most $8M$.

We conclude the desired result, $\Pro[\cF_-]\leq \frac{8\log M}{M^{1/6}}\cdot\Pro[\cF_+]=o(M^{-\frac{1}{10}})$.
\end{proof}

\begin{proof}[Proof of Lemma~\ref{lem: w bad}]
Recall that for every $t\geq 0$, we have $L\subset S_t$. Moreover, by construction of the exploration process, the component we are exploring at time $t$ has no vertices in $L$.
Thus any vertex that belongs either to the current explored component or to $V\sm S_t$, has degree at most $\frac{\sqrt{M}}{\log M}$.
This implies that, for every vertex $v$ that plays a role in the exploration process, the number of edges incident to a neighbour of $v$ is at most $\frac{M}{\log^2 M}$. This property will be crucial in our analysis.

In this proof we 
consider inputs (graphs equipped with an order of each adjacency list) instead of graphs.  As in the proof of Lemma~\ref{lem: prob next}, we will perform our switchings between equivalence classes of inputs.

We start by proving the second part of the lemma. 
We fix $w\in V\sm S_{t-1}$ and condition on the configuration $\cC_{t-1}$ at time $t-1$.
For every $0\leq i \leq d(w)-1$, we let $\cF_i$ be the equivalent class of inputs such that $w_t=w$ and that the sum of the number of  loops on $w$ and edges from  $w$ to $S_{t-1}$ is $i+1$.
For any equivalence class in $\cF_{i+1}$,
there are at least  $(i+1)(M_{t-1}-\beta M-\frac{2M}{\log^2 M})\geq \frac{2(i+1)M}{3}$ switchings that lead to one in $\cF_i$.
For any equivalence class in $\cF_i$, there are at most $8(d(w)-i)(\beta M+d(w))$ switchings that lead to $\cF_{i+1}$.
It follows that for $i+1 \ge 2\sqrt{\beta} d(w)$,
we have $\Pro[\cF_{i+1}] \le16 \sqrt{\beta}\Pro[\cF_i]$. The second statement of the lemma follows from applying the last inequality recursively. 

In the same vein, one can obtain that  conditional on $w_t\neq w$, the probability that $w$ is incident to more than $3\sqrt{\beta} d(w)$ edges connecting to $S_{t-1}$, is at most $2\sqrt{\beta}$. We omit the details. This fact will be used at the end of the proof.

As before, we fix $w\in V\sm S_{t-1}$ and condition on the configuration $\cC_{t-1}$ at time $t-1$.
We let $\cA_w$ be the union of the equivalence classes of inputs consistent 
with this configuration where $w_t=w$ and we let $\cB_w$ be those where $w_t \neq w$.

Let $\cB^i_w$ be the elements of $\cB_w$ such that there are $i$ edges between $w$ and $v_t$.
For each equivalence class in $\cB^{i+1}_w$,
we can switch any edge $v_tw$ with some other ordered edge $xy$ with $x\in V\sm S_{t-1}$ to get
an element of $\cB^i_w$ unless
$x$ is a neighbour of $v_t$ or
$y$ is a neighbour of $w$.
Thus, given a graph in  $\cB^{i+1}_w$, there are at least
$(i+1)(M_{t-1}-\frac{2M}{\log^2 M})$ switchings that lead to $\cB^{i}_w$.
On the other hand, given an element of $\cB^{i}_w$, there are at most $d(w)d(v_t) \le \frac{M}{\log^2 M}$ switchings that lead to  $\cB^{i+1}_w$.
Since $M_{t-1}\geq \frac{3M}{4}$, we have that $M_{t-1}-\frac{2M}{\log^2 M}\geq \frac{M}{2}$. This implies, $|\cB^i_w| \ge \frac{\log^2 M}{2} |\cB^{i+1}_w|$. Thus,
\begin{align}\label{eq:bound_prob}
|\cB^0_w| &\geq \left(1-\frac{1}{2\log M}\right)|\cB_w|\;.
\end{align}

We let $\cC^i_w$ be the elements of $\cB^0_w$ such that there are $i$ edges between $w_t$ and the neighbours of $w$ (recall that $w\neq w_t$).
Given an equivalence class in $\cC^{i+1}_w$,
there exist at least
$(i+1)(M_{t-1}-\beta M-\frac{2M}{\log^2 M})$ 
switchings that lead to  $\cC^{i}_w$ and that do not use any edge incident to $v_t$.
On the other hand, given an equivalence class in $\cC^{i}_w$, there are at most $(d(w_t)-(i+1))d(w)\frac{\sqrt{M}}{\log M}\leq d(w)\frac{M}{\log^2{M}}$ switchings that lead to  $\cC^{i+1}_w$.

Given that $i+1\geq \sqrt{\beta} d(w)$, we have that,
$$
|\cC^{i}_w| \geq \frac{\sqrt{\beta} d(w) \frac{M}{2}}{d(w)\frac{M}{\log^2{M}}} |\cC^{i+1}_w| \geq \log{M} |\cC^{i+1}_w|\;.
$$

Using~\eqref{eq:bound_prob}, it follows that conditional on our input being in $\cB_w$, the probability
that $w$ has no edge to  $v_t$ and  $w_t$ has at most $\sqrt{\beta} d(w)$ edges incident to $N_H(w)$ is at least $\left(1-\frac{1}{2\log M}\right)\left(1-\frac{2}{\log M}\right)\geq 1-\frac{3}{\log{M}}$.

Combining the statements above, we obtain that the proportion of elements of  $\cB_w$ such that  in the corresponding multigraph $H$ we have that  $w$ has no edge to $v_t$,  is incident to at most $3\sqrt{\beta} d(w)$ edges that are also incident to $S_{t-1}$  and $w_t$ is incident to at most $\sqrt{\beta} d(w)$ edges that are also incident to $N_H(w)$, is at least $1- 2\sqrt{\beta}-\frac{3}{\log{M}}\geq 1-3\sqrt{\beta}$. 
Note that this implies that $ww_t$ is not an edge of $H$. 

We consider switching using  an ordered pair of oriented edges $v_tw_t$ and   $wx$ of $H$  such that $x$ is not in $S_{t-1}$. For inputs as in the last paragraph,  there are at least $(1- 4\sqrt{\beta} )d(w)$ choices for an oriented edge $wx$ to switch with $v_tw_t$ to construct an input in $\cA_w$ (we simply cannot choose x in $S_{t-1}$ or such that $wx$ is an edge of $G$ and $w_tx$is an edge of $G$) . Clearly, for any input in $\cB_w$, there are at most $d(w)$  oriented edges to switch that lead to $\cA_w$.
For any equivalence class in $\cA_w$, similarly as before, there are between
$M_{t-1}-\beta M- \frac{2M}{\log^2 M}$ and $M_{t-1}$ such switchings that lead to $\cB_w$( we pick an oriented edge $xy$ of $H$ with both endpoints outside
$S_{t-1}$  and say we produced $v_tw$ and $xy$ by swapping on $v_tx$ and $wy$,
this will work for certain provided $x$ is not incident to a neighbour of $v_t$ and $y$ 
is not incident to a neighbour of $w$).

Straightforward computations give,
$$
\left(1-10\sqrt{\beta}\right)\frac{d(w)}{M_{t-1}}    \le  \Pro[\cA_w]  \le \left(1+10\sqrt{\beta}\right)\frac{d(w)}{M_{t-1}}\;.
$$
\end{proof}

\begin{proof}[Proof of Lemma~\ref{lem: hahaha1}]
Recall that $A_t=d(w_t)-\bE[d(w_t)]$ and that $B_t=d'(w_t)-\bE[d'(w_t)]$.
Note that $\bE[A_t]=\bE[B_t]=0$ and since the maximum degree of the vertices in $V(H)\sm S_0$ is at most $\frac{\sqrt{M}}{\log{M}}$, we have $|A_t|,|B_t|\leq \frac{2\sqrt{M}}{\log{M}}$.
We can apply Azuma's Inequality~(see Lemma~\ref{lem:Azuma}) to $\sum_{t'\leq t} A_{t'}$ with $N=t$ and $c_i=\frac{2\sqrt{M}}{\log{M}}$, to obtain
\begin{align*}
	\Pro\left[\sum_{t'\leq t} A_{t'}>\frac{M}{\log \log M}\right]<2e^{- \frac{M\log^2{M}}{8t(\log\log{M})^2}}<e^{-\log^{3/2}{M}}\;,
\end{align*}
since $t\leq M$ and $M$ is large enough. A union bound over all $t\leq M$ suffices to obtain that the probability there exists a $t$ such that $\sum_{t'\leq t} A_{t'}>\frac{M}{\log \log M}$ is $o(1)$. The same argument can be used for $B_t$. Thus, we obtain $\Pro[\cF_{bad}]=o(1)$.
\end{proof}

\section{Handling Vertices of Degree 2}
\label{deg2section}

\subsection{Disjoint Unions of Cycles}
\label{deg2section1}

The graph $G(\cD)$ partitions into a set of cyclic components  and a subdivision of $H({\cal D})$.
We consider first the structure of the graph formed by its cyclic components.
We let $T$ be the set of vertices in these components
and let $J(T)$ be a union of cycles chosen uniformly at random among all $2$-regular graphs with vertex set $T$.
We emphasize that $G(\cD)$ is a simple graph so these cycles have length at least $3$. 

Fix some vertex $v \in T$.
Let $p_{\ell,t}$ be the probability that $v$ is in a cycle of length $\ell$ in $J(T)$ if $t=|T|$.
Let $C_t$ be the number of configurations of $t$ vertices into disjoint cycles of length at least $3$.
We will use the following result on the asymptotic enumeration of $2$-regular graphs~(see, e.g. Example VI$.2$ in~\cite{flajolet2009analytic}).
\begin{theorem}[\cite{flajolet2009analytic}]\label{thm: Nt2}
We have
$$
C_t=\left(1+\frac{5}{8t}+O(t^{-2})\right)\frac{e^{-3/4}}{\sqrt{\pi t}}t!\;.
$$
\end{theorem}

\begin{corollary}\label{cor: pt}
For every integer $t\geq 3$ and every $3\leq \ell \leq \frac{3}{4}t$,
we have
\begin{align*}
	p_{\ell,t}=\frac{1+O(\ell t^{-2})}{2\sqrt{t(t-\ell)}}\;.
\end{align*}
\end{corollary}
\begin{proof}
If $v$ belongs to a cycle of length $\ell$,
then there are $\binom{t-1}{\ell-1}$ ways to select the remaining vertices in its cycle.
In addition, there are $\frac{(\ell-1)!}{2}$ possible configurations for the cycle containing $v$ given we have selected the vertices in this cycle.
Hence $p_{\ell,t}=\binom{t-1}{\ell-1}\frac{(\ell-1)!}{2}\frac{C_{t-\ell}}{C_t}$.
The desired results follow from straightforward computations using the bounds from Theorem~\ref{thm: Nt2}.
 \end{proof}

\begin{corollary}\label{cor: prob long cycle in v}
For every $0<\delta< \frac{3}{4}$
and for every sufficiently large $t$,
the probability that $v$ lies on a cycle of  $J(T)$
of length at least  $\frac{\delta t}{2}$ but less than $\delta t$ is at least $\frac{\delta}{5}$.
\end{corollary}
\begin{proof}
Using Corollary~\ref{cor: pt} and that $\sqrt{t(t-\ell)}<t$, we obtain that the desired probability is at least
$$
 \sum_{\ell=\frac{\delta t}{2}}^{\delta t} p_{\ell,t}
\geq \sum_{\ell=\frac{\delta t}{2}}^{\delta t} \frac{(1+o_t(1))}{2\sqrt{t(t-\ell)}}
\geq \frac{\delta}{5}\;.
$$
\end{proof}

\begin{corollary}\label{cor: prob long cycle}
For every $0<\delta< \frac{3}{8}$
and for every sufficiently large $t$,
the probability that $J(T)$ contains a cycle of length at least  $\delta t$ is at least $\frac{\delta}{3}$.
\end{corollary}
\begin{proof}
The probability there exists a cycle of length at least $\ell$ is at least
the probability that $v$ is in a cycle of length $\ell$.
Using Corollary~\ref{cor: prob long cycle in v},
we conclude that the probability that $v$ is
contained in a cycle of length at least $\delta t$ (and at most $2\delta t$) is at least $\frac{\delta}{3}$.
\end{proof}

\begin{corollary}\label{cor: prob no long cycle}
For every $0<\epsilon<1$,
there  exists a  $p_\epsilon>0$ such that for any sufficiently large $t$,
the probability that $J(T)$ contains no cycle of length at least $\epsilon t$ is at least $p_\epsilon$.
\end{corollary}
\begin{proof}
It suffices to prove the statement for $\epsilon<\frac{1}{10}$.

We let $D_\epsilon$ be the event that the sum of the lengths of the cycles of $J(T)$
which have length at least $\frac{\epsilon t}{4}$ but less than $\frac{\epsilon t}{2}$
exceeds $(1-\epsilon) t$.
Clearly it is enough to prove a lower bound on $\Pro[D_\epsilon]$.

For $k\geq 0$, we let $E_{k,\epsilon}$
be the event that $J(T)$ contains $k$  cycles $P_1,\ldots ,P_k$ of length $\ell_1,\dots,\ell_k$ such that for each $1\leq i\leq k$, if $t_i=t-\sum_{j<i} \ell_j \geq \epsilon t$, 
then $P_i$ is disjoint from $P_j$ for all $j<i$ and  the lowest indexed vertex in $T\sm \left(\cup_{j<i} P_j\right)$ is in $P_i$, and $\frac{\epsilon t}{4}\leq\ell_i\leq \frac{\epsilon t}{2}$. 

We set  $k^*=\lceil \frac{4}{\epsilon} \rceil$.
Observe that $\Pro[D_\epsilon]\geq \Pro[E_{k^*,\epsilon}]$, so 
it suffices to lower bound  $\Pro[E_{k^*,\epsilon}]$.
Clearly, $\Pro[E_{0,\epsilon}]=1$.
For $1\leq k \leq k^*$, we have $\Pro[E_{k,\epsilon}\mid E_{k-1,\epsilon}]=1$ if the number $t'$  of vertices not in the  union of the $P_j$ for $j<i$  
is less than  $\epsilon t$ as we can simply set $P_i=P_{i-1}$.  Given that $t' \ge \epsilon t$, this conditional probability is at least the probability 
the vertex $v$ with lowest index in a uniformly random disjoint union of cycles of total length $t'$  is in a cycle of length $\ell$ with $\frac{\epsilon t}{4}\leq \ell \leq \frac{\epsilon t}{2}$.
By applying Corollary~\ref{cor: prob long cycle in v} with the parameters $\delta=\frac{\epsilon t}{2t'}\leq \frac{1}{2}$  and $t'$,
we have $\Pro[E_{k,\epsilon}]\geq \Pro[E_{k,\epsilon}\mid E_{k-1,\epsilon}] \Pro[E_{k-1,\epsilon}] \geq \frac{\epsilon}{5} \Pro[E_{k-1,\epsilon}]$.
Hence $\Pro[E_{k^*,\epsilon}]\geq (\frac{\epsilon}{5})^{k^*} \geq (\frac{\epsilon}{5})^{ \lceil 4\epsilon^{-1}\rceil } $.
\end{proof}

 \subsection{The order of the  components}

 In this section we study the number of vertices in
 the union of some components of $G=G(\cD)$ conditioned on an (at least partial)  choice of $H=H({\cal D})$. 
As before, we write $M=M_\cD$. Observe that the degree sequence $\cD$ fixes the number  $m= \frac{M}{2}$ of edges of $H$  and also $n_2$, the number of vertices of degree 2 in $G$, while a choice of $H$ fixes the number of edges of $H({\cal D})$ in each of its components. 
 
For a given choice of $H$, we expose edges and vertices of degree $2$ in $G$ in two phases.
In the first phase we do the following.
\begin{enumerate}
\item[(1.1)]\label{expose:a} For each set of parallel edges between two vertices in $H$, we expose at most one of them as an edge of $G$.
The parallel edges of $H$ that we exposed as edges of $G$ will be called  \emph{fixed} edges of $H$.
\item[(1.2)]\label{expose:b} For each non-fixed edge in $H$ with distinct endpoints  which is parallel to some other edge (so it corresponds to a path in $G$ of length at least $2$), we  expose the edge on the corresponding path of $G$ which is incident to its endpoint of lowest index; that is,  we expose the other endpoint of such an edge.
\item[(1.3)]\label{expose:c} For each loop of $H$ rooted at a vertex, we expose the two edges of $G$ incident to this vertex in the cycle of $G$ corresponding to the loop; that is, we   expose the other  endpoints of these  two edges.
\end{enumerate}
We note that if an edge of $H$ is neither a loop nor parallel to another edge, then we do not expose whether it corresponds to an edge or a non-trivial path of $G$.

Let $m'$ be the number of edges of $H$ that are non-fixed.
We let $n^*_2$ be the number of vertices of degree $2$ which have been exposed in~(1.2) or~(1.3). Let $n'_2=n_2-n^*_2$ be the vertices of degree $2$ that have not been exposed yet. 
\begin{observation}\label{obs: prop}
We observe that
\begin{enumerate}[(a)]
  \item\label{obs: prop a} for every component $K$ of $H$, there are at most $\frac{|E(K)|}{2}$ fixed edges, and
  \item\label{obs: prop b} for every component $K$ of $H$, the number of vertices of degree $2$ exposed in~(1.2) or~(1.3) inside $K$, is at most $2|E(K)|$, and 
the sum of the number of such vertices and $|V(K)|$ is at most $3|E(K)|$. 
\end{enumerate}
\end{observation}

Suppose that we also condition on the set $T_H=\{v_1,\dots,v_{n^H_2}\}$ of the $n^H_2$ vertices of degree $2$ that have not been exposed yet and which lie in the non-cyclic components of $G$. 
Then we can specify the non-cyclic components of the graph $G$ in a second phase.
\begin{enumerate}
 \item[(2.1)]\label{perm:1} Fix an arbitrary ordering of the $m'$ edges of $H$ that are non-fixed and a direction on each of these edges.
 \item[(2.2)]\label{perm:2} Choose a uniformly random permutation $\pi$ of length $n^H_2+m'-1$ of $T_H$ and a  set of $m'-1$ indistinguishable delimiters.
 We let $d_i$ be the  position of the $i$-th delimiter in the permutation and add a 
 delimiter  $d_0=0$ at the start of the permutation and  a delimiter $d_{m'}=n^H_2+m'$ at its end.
 \item[(2.3)]\label{perm:3} For every $1\leq i\leq m'$, 
let $e_i=xy$ be the $i$-th non-fixed edge of $H$, with the corresponding direction. 
We let $x'$~(resp. $y'$)  be the neighbour  of $x$~(resp. $y$) on the path of $G$ 
corresponding to $e_i$ if we have exposed it, otherwise we set $x'=x$ (resp. $y'=y$). 
We expose the vertices $v_j\in T_H$ with $d_{i-1}\leq \pi(j) <d_i$ to construct a  path in $G$ connecting  $x'$ and $y'$. 
We do this by starting at $x$ and by following the order induced by $\pi$.
\end{enumerate}

Now, conditional on the choice of $T_H$ of size $n^H_2$, 
the number of choices for the non-cyclic components of $G(\cD)$  is exactly
$$
N^H(n^H_2,m')=\frac{(n^H_2+m'-1)!}{(m'-1)!}\;.
$$

Recall that, if we only condition on the information exposed in~(1.1)--(1.3), we have $m'$ non-fixed edges in $H$ and $n^*_2$ vertices of degree $2$ which were exposed in non-cyclic components. Also recall, $n'_2=n_2-n_2^*$.
For each $s,t$ with $s+t=n'_2$ and every $m'\geq 1$,
we let $N(s,t,m')$ be the number of graphs with $t$ vertices of degree $2$ in cyclic components and $s+n^*_2$ vertices of degree $2$ in non-cyclic components
given our exposure of $H$.

By the previous observations, $N(s,t,m')=\binom{s+t}{t}N^H(s,m')C_t$, where $C_t$ has been defined in the previous section as the number of configurations of disjoint cycles using $t$ vertices. Now,  $\frac{N^H(s+1,m')}{N^H(s,m')}=s+m'$.
Theorem~\ref{thm: Nt2} allows us to estimate the ratio $C_t/C_{t-1}$. We thus obtain that there exists some function $f$ such that $f(t)=O(t^{-2})$ and $0<1-f(t) < \sqrt{\frac{t}{t-1}}$, and such that for every $t\geq 4$,
\begin{align}\label{eq:crucial}
\frac{N(s,t,m')}{N(s+1,t-1,m')}
&= \frac{\binom{s+t}{t}}{\binom{s+t}{t-1}}\cdot\frac{N^H(s,m')}{N^H(s+1,m')}\cdot\frac{C_t}{C_{t-1}}
= \left(1-f(t)\right)\sqrt{\frac{t-1}{t}}\left(1-\frac{m'-1}{s+m'}\right)\;.
\end{align}
%
%\new{Furthermore, given a union of cycles using $t-1$ vertices and a choice of one of its $t-1$ edges, one can obtain a union of cycles using $t$ vertices by subdividing the edge with a new vertex $v_t$. Each union of cycles using $t$ vertices where $v_t$ is not in a triangle can be obtained in such a way.}
%%obtain a unique element of $C_{t-1}$ from every element of $C_t$
%%in which $t$ is not in a triangle, by suppressing  $t$,  and each element of $C_{t-1}$ arises in this way from exactly $t-1$  elements of $C_t$. 
%%\felix{This is a bit fishy, as $C_t$ is an integer and therefore it doesn't contain an element}
%This yields 
%$\frac{C_t}{C_{t-1}}=(t-1)(1-p_{3,t})^{-1}$.  Thus,
%
%\begin{align}
%\frac{N(s,t,m')}{N(s+1,t-1,m')}
%= \frac{(s+1)(t-1)(1-p_{3,t})^{-1}}{t(s+m')}\;.
%\end{align}
%
%Now, Since, $p_{3,4}=p_{4,t}=0$, 
%\felix{$p_{4,t=0}$? Where do we use that?}
%Lemma \ref{lem: p3} and  Corollary~\ref{cor: pt}, imply that for  some \new{function} $f$  such that $f(t)=O(t^{-2})$ and $0<1-f(t) < \sqrt{\frac{t}{t-1}}$, for every $t\geq 4$, 
%
%\begin{align}\label{eq:crucial}
%\frac{N(s,t,m')}{N(s+1,t-1,m')}
%= \left(1-f(t)\right)\sqrt{\frac{t-1}{t}}\left(1-\frac{m'-1}{s+m'}\right)\;.
%\end{align}
%\felix{I don't see the square root term in (5).}
%\guillem{I'm also confused here. The expression above is obtained by estimating the ratio $C_t/C_{t-1}$ using Theorem 29 directly. This is how it was done some time ago, and we changed it. But I don't understand why and I agree with you that using Lemma 34, you don't obtain the roots}
It is also not hard to see that for non-negative $s$ we have:
\begin{align}\label{eq:ratio}
\frac{N(s,3,m')}{N(s+3,0,m')}=\frac{1}{3!}\frac{(s+3)(s+2)(s+1)}{(s+m')(s+m'+1)(s+m'+2)}\;.
\end{align}

We let $N^*(n_2',m')=N(n_2',0,m')+\sum_{t=3}^{n_2'}  N(n_2'-t,t,m')$.
%Thus $N$ is the number of graphs with $n_2^*$ exposed vertices of degree $2$ in the first phase which are compatible with $H$.
For every $t\geq 3$, let
$$
q_t(n_2',M')= \frac{N(n_2'-t,t,m')}{N^*(n_2',m')}\;.
$$
That is, $q_t$ equals the probability that there are $t$ vertices in the cycle components given our choices of $H(\cD)$ and the exploration explained above.
Observe that if $g$ is a function such that $g(i)=O(i^{-2})$ for each $i\geq 4$ and we have $g(i)<1$, then there are two positive constants such that for every integer $j\geq 4$, the product $\prod_{i=4}^{j} \left(1-g(i)\right)$ lies between them. 
Using~\eqref{eq:crucial} and letting $f_t=\prod_{i=4}^t \left(1-f(i)\right)$, we have that for $t \ge 3$, 
\begin{align}\label{eq:vertices in cycles}
q_t& = \frac{\sqrt{3}q_3 f_t}{\sqrt{t}}  \prod_{i=4}^t\left(1-\frac{m'-1}{n_2'-t+m'}\right)\;.
\end{align}

We now provide the proof of Theorem \ref{deg2stheorem2}.
\begin{proof}[Proof of Theorem \ref{deg2stheorem2}]
Define $\alpha=\frac{\gamma}{2}$.
Using \eqref{eq:vertices in cycles} and the fact that $m'\leq M\leq b$, 
we can find positive constants $c_1$ and $c_2$ such that for every $3 \le t<(1-\frac{\alpha}{2})n'_2$, we have
\begin{align}\label{eq:vertices in cycles2}
\frac{c_1 q_3}{\sqrt{t}}\leq q_t
\leq \frac{c_2 q_3}{\sqrt{t}}\;,
\end{align} 
where we use that $n_2'$ is large in terms of $\alpha$ and $b$.
From~\eqref{eq:ratio}, we obtain that $\frac{q_0}{7}\leq q_3\leq \frac{q_0}{5}$, provided that $n_2'$ is large enough. Recall that $q_1=q_2=0$.
We also observe that $q_t\leq q_{t-1}$ for every $t\geq 4$ by \eqref{eq:crucial}. 
Thus, $q_t\leq q_{(1-\frac{\alpha}{2})n_2'}$ for every $t>(1-\frac{\alpha}{2})n_2'$.
Observe that $\frac{1}{\sqrt{(1-\frac{\alpha}{2})n_2'}}\leq\frac{\sqrt{2}}{\sqrt{t}}$ for all $t>(1-\frac{\alpha}{2})n_2'$.
Since $q_t$ is a probability distribution, it follows from~\eqref{eq:vertices in cycles2} that
\begin{align*}
 1= q_0 +\sum_{t=3}^{n_2'} q_t \leq 7q_3 +\sqrt{2}c_2 q_3 \sum_{t=3}^{n_2'} \frac{1}{\sqrt{t}}\leq 7q_3+ 3 c_2 q_3 \sqrt{n_2'}\;,
\end{align*}
from where we obtain that $q_3\geq \frac{c_1'}{\sqrt{n'_2}}$, for some positive constant $c_1'$. 
Therefore, we conclude that the probability that $t\geq (1-\alpha)n'_2$ is at least
\begin{align*}
\sum_{t=(1-\alpha)n'_2}^{(1-\frac{\alpha}{2})n'_2} q_t 
& \geq \frac{c_1 c_1'}{\sqrt{(1-\frac{\alpha}{2})}\cdot n'_2}\cdot \frac{\alpha n'_2}{2} 
=  \frac{c_1 c_1'\alpha }{2\sqrt{(1-\frac{\alpha}{2})}} =:\delta^*\;.
\end{align*}

Now we are able to conclude the proof. 
Recall that $\alpha=\frac{\gamma}{2}$. 
Observe that $n'_2= n_2-n^*_2\geq n-2M$, since $n^*_2\leq M$ and $n_2\geq n-M$.
It follows that with probability $\delta^*$ we have $t\geq (1-\alpha)n'_2\geq (1-\frac{2\gamma}{3}) n$. 
If this is the case, there are at most $\frac{2\gamma}{3}n$ vertices in non-cyclic components, and thus, there is no component of order at least $\gamma n$ in such components.
First, we apply Corollary~\ref{cor: prob long cycle} with $t\geq (1-\frac{2\gamma}{3})n \geq \frac{n}{2}$ and $\delta= 2 \gamma$, and obtain that the probability that there exists a cycle of length at least $2\gamma t\geq \gamma n$ is at least $\delta_1>0$.
Second, we apply Corollary~\ref{cor: prob no long cycle} with $t\geq (1-\frac{2\gamma}{3})n$ and $\epsilon= \gamma$, and obtain that the probability that there exists no cycle of length at least $\gamma t\leq \gamma n$ is at least $\delta_2>0$. Let $\delta'=\min\{\delta_1,\delta_2\}$.
Finally, since this holds for every conditioning on $H$, $m'$ and $n_2^*$, it also holds for the unconditioned
statement and we have proved the theorem for $\delta=\delta^* \delta'$.
\end{proof}
We finish this section with two results that will be useful to proof Theorem~\ref{deg2stheorem1} and~\ref{deg2stheorem}.
\begin{lemma}\label{lem: deg 2 in H}
For any positive constant $\beta$,  if  $n_2'\geq \beta M$,
 then  the probability that there are more than $\frac{n'_2}{\sqrt{M}}$ vertices in cyclic components of $G$ is $o_M(1)$.
\end{lemma}
\begin{proof}
Recall that by Observation~\ref{obs: prop}~(a), we have $\frac{M}{4}\leq m'\leq M$. 
We split the proof into two cases.

First, suppose that $n_2'\leq 2M$. 
We use~\eqref{eq:vertices in cycles} to upper bound $q_t$ and obtain the desired probability. 
There exists some constant $c>0$ such that
\begin{align*}
\sum_{t\geq \frac{n'_2}{\sqrt{M}}}q_t 
&\leq c \sum_{t\geq \frac{n'_2}{\sqrt{M}}} \left(1-\frac{m'-1}{n'_2+m'}\right)^{t-3}
\leq c \sum_{t\geq \frac{n'_2}{\sqrt{M}}} \left(1-\frac{M/4}{2M+M/4}\right)^{t-3}\\
&\leq c \sum_{t\geq \frac{n'_2}{\sqrt{M}}} \left(\frac{8}{9}\right)^{t-3}
\leq 9c \left(\frac{8}{9}\right)^{n'_2/\sqrt{M}-3} = o_M(1)\;,
\end{align*}
since $n_2'\geq \beta M$.

Now, suppose that $n_2'\geq 2M$. We use~\eqref{eq:vertices in cycles} to lower bound $q_t$ for every $t\leq \frac{n'_2}{10}$. There exists some constant $c>0$ such that,
$$
q_t\geq \frac{c q_3}{\sqrt{n'_2}}\left(1-\frac{m'-1}{n'_2-\frac{n'_2}{10}+m' } \right)^t\geq \frac{c q_3}{\sqrt{n'_2}}\left(1-\frac{3M}{2n'_2}\right)^t\;. 
$$
Since $q_t$ is a probability distribution and since $n_2'\geq 2M$, we have
\begin{align*}
1&\geq \sum^{n'_2/10}_{t=3} q_t \geq  \frac{c q_3}{\sqrt{n'_2}} \sum_{t=3}^{n'_2/10} \left(1-\frac{3M}{2n'_2}\right)^t 
\geq   \frac{c q_3}{\sqrt{n'_2}}\cdot \frac{(1-\frac{3M}{2n'_2})^3-(1-\frac{3M}{2n'_2})^{n'_2/10+1}}{1- (1-\frac{3M}{2n'_2})} = \frac{c' q_3 \sqrt{n'_2}}{M}\;,
\end{align*}
for some $c'>0$, from which we conclude that $q_3\leq \frac{M}{c'\sqrt{n'_2}}$.
Now we use again~\eqref{eq:vertices in cycles} to upper bound the desired probability
\begin{align*}
\sum_{t\geq \frac{n'_2}{\sqrt{M}}}q_t \leq \frac{c q_3}{\sqrt{n'_2/\sqrt{M}}} \sum_{t\geq \frac{n'_2}{\sqrt{M}}} \left(1-\frac{M/4-1}{n'_2+M/4}\right)^t
\leq \frac{c M^{5/4}}{c' n'_2} \sum_{t\geq \frac{n'_2}{\sqrt{M}}} \left(1-\frac{M}{8n'_2}\right)^t \\
\leq \frac{c M^{5/4}}{c' n'_2} \cdot\frac{\left(1-\frac{M}{8n'_2}\right)^{\frac{n'_2}{\sqrt{M}}}}{1- \left(1-\frac{M}{8n'_2}\right)} 
\leq \frac{8c}{c'}\cdot M^{1/4} e^{-\sqrt{M}/8}= o_M(1)\;.
\end{align*}
\end{proof}

%
%
%\begin{proof}
%Recall that by Observation~\ref{obs: prop}~(a), we have $\frac{M}{4}\leq m'\leq M$. Moreover, we have $n_2'\geq n-2M\geq 2M$.
%
%
%
%We use~\eqref{eq:vertices in cycles} to lower bound $q_t$ for every $t\leq n'_2/10$. There exists some constant $c>0$ such that,
%$$
%q_t\geq \frac{c q_3}{\sqrt{n'_2}}\left(1-\frac{m'-1}{n'_2-\frac{n'_2}{10}+m' } \right)^t\geq \frac{c q_3}{\sqrt{n'_2}}\left(1-\frac{3M}{2n'_2}\right)^t\;. 
%$$
%Since $q_t$ is a probability distribution, we have
%\begin{align*}
%1&\geq \sum^{n'_2/10}_{t=3} q_t \geq  \frac{c q_3}{\sqrt{n'_2}} \sum_{t=3}^{n'_2/10} \left(1-\frac{3M}{2n'_2}\right)^t 
%\geq   \frac{c q_3}{\sqrt{n'_2}}\cdot \frac{(1-\frac{3M}{2n'_2})^3-(1-\frac{3M}{2n'_2})^{n'_2/10+1}}{1- (1-\frac{3M}{2n'_2})} = \frac{c' q_3 \sqrt{n'_2}}{M}\;,
%\end{align*}
%for some $c'>0$, from which we conclude that $q_3\leq \frac{M}{c'\sqrt{n'_2}}$.
%
%Now we use again~\eqref{eq:vertices in cycles} to upper bound the desired probability.
%\begin{align*}
%\sum_{t\geq \frac{n'_2}{\sqrt{M}}}q_t \leq \frac{c q_3}{\sqrt{n'_2/\sqrt{M}}} \sum_{t\geq \frac{n'_2}{\sqrt{M}}} \left(1-\frac{M/4-1}{n'_2+M/4}\right)^t
%\leq \frac{c M^{5/4}}{c' n'_2} \sum_{t\geq \frac{n'_2}{\sqrt{M}}} \left(1-\frac{M}{8n'_2}\right)^t \\
%\leq \frac{c M^{5/4}}{c' n'_2} \cdot\frac{\left(1-\frac{M}{8n'_2}\right)^{\frac{n'_2}{\sqrt{M}}}}{1- \left(1-\frac{M}{8n'_2}\right)} 
%\leq \frac{8c}{c'}\cdot M^{1/4} e^{-\sqrt{M}/8}= o_M(1)\;.
%\end{align*}
%\end{proof}

\begin{lemma}
\label{lilyliveredlem} For every  positive constant $\beta<\frac{1}{100}$, if $n'_2\geq\beta M$ the following is satisfied. Fix a choice of $H$ and let $U_H$ be a union of some components of $H$ with $|E(U_H)|\geq \beta M$. Let $U_G$ be the union of the corresponding components of $G$. Then, with probability $1-o_M(1)$,
$$\frac{|E(U_H)|}{8M}\cdot n'_2\leq |V(U_G)|\leq \frac{16|E(U_H)|}{M}\cdot n'_2+4|E(U_H)|\;.
$$ 
\end{lemma}
\begin{proof}

Observe that the choice of $H$  determines  $m'$, the number of edges that have not been fixed in~(1.1), and $n^*_2$, the number of vertices of degree 2 which have been exposed in~(1.2) or~(1.3). We will also condition on $n_2^H$, the number of vertices of degree $2$ that have been exposed in~(2.3). 
Since $n_2'\geq\beta M$, by Lemma~\ref{lem: deg 2 in H}, the probability that  $n_2^H\leq \frac{n'_2}{2}$ is $o_M(1)$. 
Hence, we may assume that  $n_2^H\geq \frac{n'_2}{2}$.

We denote by $n_2^*(U_H)$ the number of vertices of degree $2$ exposed in~(1.2) or~(1.3) in a component of $U_H$.
By Observation~\ref{obs: prop}~\eqref{obs: prop b}, we have 
$$
0\leq n_2^*(U_H)\leq 2|E(U_H)|\;.
$$
We let $m'(U_H)$ be the number of non-fixed edges in $U_H$. By Observation~\ref{obs: prop}~\eqref{obs: prop a}, we have $\frac{|E(U_H)|}{2}\leq m'(U_H)\leq |E(U_H)|$. Similarly, $\frac{M}{4}=\frac{m}{2}\leq m'\leq m=\frac{M}{2}$. Thus,
$$
\frac{|E(U_H)|}{M}\leq \frac{m'(U_H)}{m'}\leq \frac{4|E(U_H)|}{M}\;.
$$
Let $n_2^H(U_H)$ be the number of vertices of degree $2$ which have been exposed in~(2.3) to the edges of $U_H$. Since the ordering of the edges in (2.1) was arbitrary, symmetry amongst the non-fixed edges yields: 
$$
\frac{n_2'|E(U_H)|}{2M}\leq \frac{n_2^H|E(U_H)|}{M}\leq \bE[n_2^H(U_H)]= \frac{n_2^H m'(U_H)}{m'}\leq \frac{4n_2^H|E(U_H)|}{M} \leq \frac{4n'_2|E(U_H)|}{M} \;.
$$
Since the minimum degree in $H$ is at least one, there are at most $2|E(U_H)|$ vertices in $U_H$. So, the number of vertices in $U_G$ satisfies
$$
n_2^H(U_H) \leq |V(U_G)|=|V(U_H)|+n_2^*(U_H)+\bE[n_2^H(U_H)] \leq n_2^H(U_H) +4|E(U_H)|\;.
$$
Now, we will use our random permutation model to show that the random variable $n_2^H(U_H)$ is concentrated around its expected value, which by the previous equation implies that $|V(U_G)|$ also is. In~(2.1) we insist on choosing an ordering of the non-fixed edges of $H$ in such a way that the $m'(U_H)$ first ones correspond to $E(U_H)$.

The probability that $n_2^H(U_H)\geq \ell$ conditional on the value of $n_2^H$, is the same as the probability that in choosing $m'-1$ elements in $\{1,2,...,(m'-1+n_2^H)\}$, we choose less than $m'(U_H)$ of them
that are smaller than $m'(U_H)+\ell$. So, since $n_2^H \geq \frac{n'_2}{2}\geq  \frac{\beta M}{2}$, letting  $X_{\ell}$ be the number of elements chosen from the $m'-1$ ones which are smaller than $m'(U_H)+\ell$,  a standard concentration argument on $X_\ell$ for $\ell=4\bE[n_2^H(U_H)]$~\footnote{Observe that $X_\ell$ follows a hypergeometric distribution.}, shows that with probability $1-o_M(1)$ we have $n_2^H(U_H)\leq 4\bE[n_2^H(U_H)]$.
The same holds for the probability of the event $n_2^H(U_H)\geq \frac{\bE[n_2^H(U_H)]}{4}$, concluding the proof of the lemma. 
\end{proof}

\section{Relating the size of a component of $H(\cD)$ to the order of a component of $G(\cD)$}\label{sec:relating}

As usual we use $M=M_\cD$ and $R=R_\cD$ throughout this section.
We  start with the proof of Theorem~\ref{deg2stheorem1} and conclude the section with the proof of Theorem~\ref{deg2stheorem}.
\begin{proof}[Proof of Theorem~\ref{deg2stheorem1}]
By Theorem \ref{thm: main4} with $\epsilon=1/3$, there is a $\tilde \gamma$ such that if $R \ge  \frac{M}{3}$,
then the probability that $H$ has no component of size $\tilde \gamma M$ is $o(1)$. We will 
choose $\rho$ to be the minimum of $\frac{\gamma}{30}$ and $\tilde\gamma$. 
Hence we can assume that $R \le  \frac{M}{3}$. 

To complete the proof, we show that under this assumption, the probability $G$ has a component of order $30\rho n$ given $H$ has no component of size $\rho M$ is $o(1)$. Under this hypothesis, for every component $K$ of $H$, we have $|E(K)|\leq \rho M$. Since each component is connected, this also implies that $|V(K)|\leq \rho M+1$.

The following claim allow us to bound $M$ in terms of $n$.
\begin{claim}\label{cla: R n2}
If $n_2$ is the number of vertices of degree $2$ in $\cD$, then
$$
R\geq M-2(n-n_2)\;.
$$
\end{claim}
\begin{proof}
Suppose that $R<M-2(n-n_2)$. By the definition of~$j_\cD$, we obtain $\sum_{k=1}^{j_\cD-1} d_{\pi_k}\geq M+2n_2-R> 2n$.
Since the function $f(x)=x(x-2)$ is convex,
$$
\sum_{k=1}^{j_\cD-1} d_{\pi_k}(d_{\pi_k}-2)> (j_\cD-1) \frac{2n}{j_\cD-1}\left(\frac{2n}{j_\cD-1}-2 \right)>0,
$$
which is a contradiction to the choice of $j_\cD$. Thus $R\geq M-2(n-n_2)$.
\end{proof}
Since $R\leq \frac{M}{3}$, Claim~\ref{cla: R n2} implies $M\leq 3n$.
\vspace{0.3cm}

Condition now on the choice of  $H$ and on the set of fixed edges and let $m'$ be the number of non-fixed ones.  These choices  determine $n_2'$.

Suppose that $n'_2 \leq \rho n$ and recall that $n_2^*(K)\leq 2|E(K)|$ by Observation~\ref{obs: prop}. Then, deterministically, each component of $G$ has at most $|V(K_H)|+n^*_2(K)+n'_2\leq 4\rho M+1 \leq 13 \rho n$ vertices.

Otherwise, $n'_2\geq \rho n\geq \frac{\rho M}{3}$. Group the components of $H$ in sets $U_1,\dots, U_s$ such that $\frac{\rho M}{2}\leq |E(U_i)|\leq \rho M$. Clearly such a collection exists and $s\leq 2(\rho)^{-1}$.
For every $1\leq i\leq s$, we apply Lemma~\ref{lilyliveredlem} with $U_H=U_i$ and $\beta=\frac{\rho}{4}$. 
Since $\cD$ is well-behaved, the conclusion of Lemma~\ref{lilyliveredlem} holds with probability $1-o(1)$.
Using a union bound over all the sets $U_i$, we obtain that with probability $1-o(1)$, for $1\leq i\leq s$, we have
$$
|V(U_i)|\leq \frac{16|E(U_i)|}{M}\cdot n'_2 +4|E(U_i)|\leq 16\rho n+4\rho M \leq 28\rho n\;.
$$
\end{proof}

\begin{proof}[Proof of Theorem~\ref{deg2stheorem}]
It is enough to prove the theorem for sufficiently small positive  $\rho$,
for which we will prove it with $\gamma=\rho^3$.

We can use Lemma~\ref{lem: dense case} to show that if $M\geq n\log \log{n}$, then
the probability $G$ has no component of order $(1-o(1))n$ is $o(1)$.

Next, suppose that $M \leq \frac{n}{3}$. We show that in this case, the  probability that $G$ has no component of order $\rho^3 n$ given that 
$H$ contains a component  of size $\rho M$ is $o(1)$. 
We have  that $n_2'\geq n-2M\geq \frac{n}{3} \geq M$.
Letting $K_H$ be the component of $H$ of size at least $\rho M$ and  $K_G$ be the component of $G$ corresponding to $K_H$, and applying
Lemma~\ref{lilyliveredlem} with $U_H=K_H$ and $\beta=\rho$, we conclude that with probability $1-o(1)$ (since $\cD$ is well-behaved)
$$
|V(K_G)| \geq \frac{|E(K_H)|}{8M}\cdot n'_2 \geq \frac{\rho n}{24}\geq \rho^3 n\;.
$$

Thus, it remains to prove the theorem for
$$
\frac{n}{3}\leq M\leq n\log\log{n}\;.
$$
For any subgraph $K$ of $G$, let $\ex(K)=|E(K)|-|V(K)|$ be the excess of  $K$. 
We also define the near-excess of  $K$ as $\nex(K)= \ex(K)+|L\cap V(K)|$,
where $L$ is the set of vertices of degree at least $\frac{\sqrt{M}}{\log{M}}$ in $K$. Let $S$ be the set of vertices of $G$ that correspond to a component $K$ in $H$ with $\nex(K)\geq \frac{\rho M}{2}$.

The following claim is crucial to finish the proof of the theorem.
\begin{claim}\label{cla:82}
 We have
 $$
 \Pro\left[ S \neq \emptyset, |S| < 3\rho^2n\right]=o(1)\;.
 $$
\end{claim}
We now conclude the proof of the theorem, assuming that the claim is true.
Suppose $H$ contains  a component  $K_H$ that satisfies $|E(K_H)|\geq \rho M$. If the corresponding component $K_G$  in $G$ satisfies 
$|V(K_G)| \ge \rho^3 n$, there is nothing to prove. So, suppose  
$|V(K_G)| < \rho^3 n$.  Then
\begin{align}\label{eq:ex}
\nex(K_H) \geq \ex(K_H) = |E(K_H)|- |V(K_H)| \geq \rho M- \rho^3 n\geq \frac{\rho M}{2}\;,
\end{align}
and $S$ is non-empty.

Since the total excess of the graph is at most $M$, the total near-excess is at most $M+|L|\leq 3M/2$ and there are at most $\frac{3}{\rho}$ components in $S$. Hence, there exists a component in $G$ with at least $\frac{\rho |S|}{3}$ vertices, which by Claim~\ref{cla:82} implies that with probability $1-o(1)$, $G$ has a component of order at least $\rho^3 n$.
\vspace{0.2cm}

So, it is indeed enough to prove  Claim~\ref{cla:82}.
To do so, we need the following: 

\begin{claim}\label{cla:81}
 We have
 $$
 \Pro\left[S \neq \emptyset, ~~ \sum_{v\in L\sm S} d(v)> M^{2/3}\right]=o(1)\;.
 $$
\end{claim}

\begin{proof}[Proof of Claim~\ref{cla:81}]

We let $\cA$ be the event that $S\not= \emptyset$,
and for any vertex $v\in L$,
we let
$\cB_v$ be the event that $S\not= \emptyset$, but $v\notin S$.

Say a graph $G$ is in $\cB_v$.
We only consider switchings between ordered pairs of oriented edges $vx$ and $yz$ in $H$,
where $yz$ is an edge in a component of $H$ whose vertex set is in  $S$, which is not a cut-edge and we allow that $v=x$ or $y=z$.

Since $yz$ is in a component of $S$ and $S\not=\emptyset$, 
there are at least $\ex(S)\geq \frac{\rho M}{3}$ choices for $yz$. 
Clearly, there are $d(v)$ choices for the (directed) edge $vx$.

We show below that there are at most $4M$ such switchings from any element of $\cA\sm\cB_v$ to $\cB_v$.
Thus
$$
\Pro[\cB_v] \le \frac{12 \Pro[\cA\sm \cB_v] }{\rho d(v)}\leq \frac{12 }{\rho d(v)}.
$$
So, we have,
$$
\bE\left[\sum_{S\neq \emptyset,v\in L\sm S} d(v)\right] = \sum_{v\in L} d(v)\Pro[\cB_v] \leq 12 \rho^{-1}|L|\;.
$$
Since $|L|\leq \sqrt{M}\log{M}$, Markov's inequality implies that
$$
\Pro\left[S \neq \emptyset, \sum_{v\in L\sm S} d(v)>M^{2/3}\right]=o(1)\;,
$$
and the claim follows.

It remains to prove the mentioned bound on the number of switchings between $\cA\sm\cB_v$ and $\cB_v$.
In doing so, we note that (i) if a connected subgraph $J$ contains a subgraph of near-excess $a$, 
then $J$ also has near-excess at least $a$, and (ii) the near-excess of the disjoint  union of $J_1$ and $J_2$ is at least the sum of the near-excesses of $J_1$ and $J_2$.  

Consider a graph $G$ that is in $\cA\sm\cB_v$.
Let $K$ be the component of $H$  containing $v$
and let $K_1,\ldots,K_\ell$ be the components of $K - (N(v) \cup \{v\})$.
For each neighbour $w$ of $v$,
let $K_w$ be the graph induced by $w$ and all components in $\{K_1,\ldots,K_\ell\}$ in which $w$ has a neighbour.

In the following,  we consider a triple of vertices $x,y,z$  such that 
switching $\{vy,xz\}$ leads to a graph in $\cB_v$ such that the edge $yz$ is not a cut-edge and the component containing $y$ and $z$ has near-excess at least $\frac{\rho M}{2}$. We note that unless $v=x$, this implies there is at most one edge between $v$ and $y$. 
We note further  that $v,x$ are distinct from $y,z$ but $z$ and $y$ may coincide as may 
$v$ and $x$.

If $v=x$, then $y$ and $z$ are both neighbours of $v$ in $H$ and there are exactly 
two edges of $H$ from $H-K_z-K_y$ to $K_z \cup K_y$. Furthermore, there is 
a path of $K_z \cup K_y$ from $z$ to $y$ and  the near-excess of $K_z \cup K_y$ is at least $\frac{\rho M -2}{2}$. So, we see that there are at most two choices for the 
pair $y,z$ and at most eight choices for switchings of this type. 

If $v \neq x$, then $z$ must be in $K_y$ and in $H-xz$ there can be no path from 
$z$ to any of $N(v)\sm\{ y\}$. 
Thus, there is a unique choice of $vy$ for each such ordered $xz$ and at most $2M-2d(v)$ total choices for such switchings. 
\end{proof}

\begin{proof}[Proof of Claim~\ref{cla:82}]
Let $Z$ be the number of sets of vertices $T$ that satisfy
\begin{enumerate}[(1)]
\item\label{cond:cla1} $|T|\leq 3\rho^2 n$,
\item\label{cond:cla2} the sum of the degrees of the elements in $T$ is at least  $\rho M$,
\item\label{cond:cla3} $\sum_{v\in L\sm T} d(v)\leq M^{2/3}$,
\item\label{cond:cla4} there are no edges between $T$ and $V(H)\sm T$.
\end{enumerate}
Observe that $S$ as defined satisfies~\eqref{cond:cla2}, since it is non-empty and each component in $S$ has at least $\frac{\rho M}{2}$ edges. Since $S$ is a union of components, it also satisfies~\eqref{cond:cla4}.

We will show that $\bE[Z]=o(1)$. This directly implies that the probability $S$ satisfies~\eqref{cond:cla1} and ~\eqref{cond:cla3} is at most $o(1)$,  which combined with Claim~\ref{cla:81} yields  Claim~\ref{cla:82}.

Fix a set of vertices $T$ that satisfies~\eqref{cond:cla1},~\eqref{cond:cla2} and~\eqref{cond:cla3}.
For every $0\leq k\leq k^*:=2\lceil \frac{\rho n}{26} \rceil$, let $\cA_{T,k}$ be event that there are exactly
$k$ edges that connect $T$ with $V(H)\sm T$. 

There are at most $k^2$ switchings from  a graph in $\cA_{T,k}$ which yields a graph  in $\cA_{T,k-2}$. 

We now lower bound the number of switchings from a graph in $\cA_{T,k-2}$, to a graph  $\cA_{T,k}$.
To do so we consider pairs consisting of   (i) an edge $xy$ such that  $x\notin T$  and $x$ is not a neighbour of a vertex in $T\cup L$,   and  (ii) an edge $uv$ with both endpoints in $T$ such that $y$ is not adjacent simultaneously to both $u$ and $v$. 
For such a pair, we can always switch $xy$ with at least one of  $uv$ or $vu$. 
There are at least $n-|T|-(k-2)-M^{2/3}>\frac{n}{2}$ choices for $x$, since $T$ satisfies~\eqref{cond:cla3} and $M\leq n\log\log{n}$. 
Since $d(x)\geq 1$, there are at least the same number of choices for the edge $xy$.
Since $y\notin L$, we have $d(y)\leq \frac{\sqrt{M}}{\log{M}}$ which implies that there  are at most $\left(\frac{\sqrt{M}}{\log{M}}\right)^2=\frac{M}{\log^2{M}}$ edges within $T$ that have both endpoint adjacent to $y$.
Using that $T$ satisfies~\eqref{cond:cla2} and that $n \le 3M$, 
we conclude that there are at least $\frac{\rho M}{2}-(k-2)- \frac{M}{\log^2{M}}\geq \frac{\rho M}{4}$ choices for the edge $uv$, given the choice of $xy$. The total number of switchings is at least $\frac{\rho n M}{8}$.

So, using that $k\leq k^*$ and that $n\leq 3M$
$$
\Pro[\cA_{T,k-2}]\leq \frac{k^2}{\frac{\rho n M}{8}}\Pro[\cA_{T,k}]\leq \frac{\rho n}{3M}\Pro[\cA_{T,k}]\leq \rho \Pro[\cA_{T,k}]\;.
$$
Therefore,
$$
\Pro[\cA_{T,0}]\leq \rho^{\frac{k^*}{2}}\Pro[\cA_{T,k^*}]\leq \rho^\frac{\rho n}{26};.
$$
Since $T$ satisfies~\eqref{cond:cla1}, there are at most $\binom{n}{3\rho^2n}$ 
choices for $T$.  
Provided that $\rho$ is small enough, we conclude that $\bE[Z]=o(1)$.
\end{proof}
This completes the proof of Theorem~\ref{deg2stheorem}.
\end{proof}

\section{Applications of Theorem~\ref{sequencetheorem}}
\label{sec.imp}

We briefly show that Theorem~\ref{sequencetheorem} implies the results mentioned in Section~\ref{sec:previous results}.
We consider a sequence of degree sequence $\fD=(\cD_n)_{n\geq 1}$ 
such that 
\begin{enumerate}[(d.1)]
	\item it is feasible, smooth, and sparse,
	%\item \new{$\lambda_i:= \lim_{n\to \infty} \frac{|\{j:\, d^{(n)}_j=i\}|}{n}$ exists, for every $i\geq 1$,} 
	\item $\lambda_2<1$, and
	\item $\lambda=\sum_{i\geq 1} i\lambda_i$.
\end{enumerate}
Conditions (d.1)--(d.3) are essentially $\BR$-conditions and they are weaker than $\MR$-conditions and $\JL$-conditions. An interesting consequence of them is the following: 
for every $\epsilon>0$, there exist positive integers $C$ and $n_\epsilon$ such that if $n\geq n_\epsilon$, then
$$
\sum_{i\geq C} i n_i \leq \epsilon n\;.
$$
Therefore, any sequence of degree sequences that satisfies (d.1)--(d.3) has almost all the edges incident to vertices of bounded degree.

The results of Molloy and Reed~\cite{molloy1995critical}, Janson and Luczak~\cite{janson2009new}, and Bollob\'as and Riordan~\cite{bollobas2015old} on the existence of a giant component, are of the following form: provided that $\fD$ satisfies certain conditions, 
if $Q(\fD)>0$ then $G(\cD_n)$ has a linear order component with probability $1-o(1)$, 
and if $Q(\fD)\leq 0$, then $G(\cD_n)$ has no linear order component with probability $1-o(1)$.\footnote{In fact, the Molloy-Reed result does not discuss the case $Q(\fD)=0$.}
We will show that if $\fD$ satisfies (d.1)--(d.3), then Theorem~\ref{sequencetheorem} implies the same statement.
\begin{theorem}\label{thm:stronger}
 Let $\fD=(\cD_n)_{n\geq 1}$ be a sequence of degree sequences that satisfies conditions (d.1)--(d.3). Then,
 \begin{enumerate}
  \item if $Q(\fD)>0$, then there exists a constant $c_1>0$ such that the probability that $G(\cD_n)$ has a component of order at least $c_1 n$ is $1-o(1)$.
  \item if $Q(\fD)\leq 0$, then for every constant $c_2>0$, the probability that $G(\cD_n)$ has no component of order at least $c_2 n$ is $1-o(1)$.
 \end{enumerate}
\end{theorem}
\begin{proof}
First of all, observe that (d.2) implies that $\fD$ is well-behaved.

Fix $\delta>0$ small enough. 
Since  $\lambda=\sum_{i\geq 1} i \lambda_i$, 
there is an integer $K$ such that $\sum_{i=1}^{K} i\lambda_i \geq \lambda - \delta$.
Also, since $\fD$ is smooth and provided that $n$ is large enough, 
we have
\begin{align*}
	\sum_{i=1}^{K}i\left|\lambda_i- \frac{n_i}{n}\right| < \delta\;,
\end{align*}
and hence, as $\fD$ is sparse,
\begin{align}\label{eq: large deg ver}
	\sum_{i> K}i \cdot\frac{n_i}{n} < 3\delta \;.
\end{align}

Recall $Q=Q(\fD)=\sum_{i\geq 1}i(i-2)\lambda_i$. Assume first that $Q>0$. 
We will show that $\fD$ is lower bounded.
Fix $\gamma>0$ such that $Q>\gamma$.
Now note that there exists a positive integer $k$ such that $\sum_{i= 1}^k i(i-2)\lambda_i> \frac{\gamma}{2}$.
Since $\fD$ is smooth and provided that $n$ is large enough,
we have that for every $1\leq i \leq k$
\begin{align*}
	\left|i(i-2) \frac{n_i}{n}-i(i-2)\lambda_i\right|< \frac{\gamma}{4k}.
\end{align*}
Therefore,
\begin{align*}
	\sum_{i=1}^k i(i-2) \frac{n_i}{n}
	\geq \sum_{i=1}^k i(i-2) \lambda_i -  \sum_{i=1}^k \left|i(i-2)\frac{n_i}{n}-i(i-2)\lambda_i \right|
	\geq \frac{\gamma}{4}.
\end{align*}
Since every vertex of degree $i\in [k]$ contributes in at most $i(k-2)$ to the previous sum, this implies 
$R_{\cD_n}\geq \frac{\gamma}{4k}\cdot n.$
Using that $\fD$ is sparse and since $n$ is large enough, we have that $\sum_{i\geq 1}i n_i\leq 2\lambda n$
and hence $M_{\cD_n}\leq 2\lambda n$.
Therefore, 
$$
R_{\cD_n}\geq \frac{\gamma}{8\lambda k}\cdot M_{\cD_n}
$$
and thus, the sequence $\cD_n$ is lower-bounded with $\epsilon=\frac{\gamma}{8\lambda k}$. Since it is also well-behaved, Theorem~\ref{sequencetheorem}
implies that there exists a constant $c_1>0$ such that the probability that $G$ has a component of order at least $c_1 n$ is $1-o(1)$.

Now suppose that $Q\leq0$. 
We aim to show that $\fD$ is upper bounded; 
that is, for every sufficiently small $\epsilon>0$ and large enough $n$, 
we have $R_{\cD_n}\leq \epsilon M_{\cD_n}$. 
We fix an arbitrary  and sufficiently small $\epsilon>0$.
Observe first that $\sum_{i=1}^ki(i-2)\lambda_i\leq 0$ for any positive integer $k$, as $i(i-2)\geq 0$ for  every $i\geq 2$.

Furthermore,
for sufficiently large $n$,
the number of vertices of degree different than $2$ is at least $\frac{1- \lambda_2}{2}n$.
Since the minimum degree is at least one, $M_{\cD_n}\geq \frac{1- \lambda_2}{4}n$.

Using~\eqref{eq: large deg ver}, we can choose $K$ large enough such that
$$
\sum_{i> K}i \cdot\frac{n_i}{n}< \frac{(1-\lambda_2)\epsilon}{8}\;.
$$

As before, since $\fD$ is smooth, 
we can consider $n$ large enough such that for every $1\leq i\leq K$, 
we have  $|i(i-2)\frac{n_i}{n}-i(i-2)\lambda_i|< \frac{(1-\lambda_2)\epsilon}{8K}$.
Therefore,
\begin{align*}
	\sum_{i=1}^K i(i-2)\frac{n_i}{n} 
	\leq \sum_{i=1}^K i(i-2) \lambda_i + \sum_{i=1}^K \left|i(i-2)\frac{n_i}{n}-i(i-2)\lambda_i \right|
	\leq 0 +\frac{(1-\lambda_2)\epsilon}{8}\;.
\end{align*}
Because $i\leq i(i-2)$ for every $i\geq 3$,
this implies
\begin{align*}
	R_{\cD_n}
	&\leq \frac{(1-\lambda_2)\epsilon}{8}n+\sum_{i> K} i \cdot n_i\\
	&\leq \frac{(1-\lambda_2)\epsilon}{4}\cdot n\\
	&\leq \epsilon M_{\cD_n}\;.
\end{align*}
Note that the choice of $\epsilon$ was arbitrary and thus $\fD$ is upper-bounded. 
Since it is also well-behaved, Theorem~\ref{sequencetheorem}
implies that for every constant $c_2>0$, the probability that $G$ has no component of order at least $c_2 n$ is $1-o(1)$.
\end{proof}

Theorem~\ref{sequencetheorem} can be also applied to obtain results on the existence of a giant component in specific models of random graphs. Here, as an example, we will study the case of the Power-Law Random Graph introduced by Aiello, Chung and Lu~\cite{aiello2000random}.
Let us first recall its definition. 
Choose two parameters $\alpha,\beta>0$ and consider the sequence of degree sequences $\fD=(\cD_n)_{n\geq 1}$ where $\cD_n$ has 
 $n_i(n)=\lfloor e^{\alpha} i^{-\beta}\rfloor$ vertices of degree $i$, for every $i\geq 1$.  We should think about these parameters as follows: $\alpha$ is typically large and determines the order of the graph (we always have $\alpha=\Theta(\log{n})$), and $\beta$ is a fixed constant that determines the power-decay of the degree sequence.
 The total number of vertices, can be determined in terms of $\alpha$ and $\beta$,
 $$
  n \approx
  \begin{cases}
    \zeta(\beta) e^\alpha        & \quad \text{if } \beta>1 \\
    \alpha e^\alpha        & \quad \text{if } \beta=1 \\
    \frac{e^{\frac{\alpha}{\beta}}}{1-\beta}         & \quad \text{if } 0<\beta<1 \\
  \end{cases}
 $$
 and similarly for the number of edges
$$
  m \approx
  \begin{cases}
    \frac{1}{2}\zeta(\beta-1) e^\alpha        & \quad \text{if } \beta>2 \\
    \frac{1}{4}\alpha e^\alpha        & \quad \text{if } \beta=2 \\
    \frac{1}{2}\frac{e^{\frac{2\alpha}{\beta}}}{2-\beta}       & \quad \text{if } 0<\beta<2 \\
  \end{cases}
 $$
 where $\zeta(x)=\sum_{i\geq 1} i^{-x}$ is the standard zeta function.
 Moreover, the maximum degree is $e^{\alpha/\beta}$.
 
 The \emph{Power-Law Random Graph}, denoted by $G=G(\alpha,\beta)$, is a graph on $n$ vertices chosen uniformly at random among all such graphs with degree sequence $\cD_n$.
There are some values of $\beta$ (for instance $\beta=1$) for which $\fD$ does not satisfy the conditions (d.1)--(d.3). Thus, Theorem~\ref{thm:stronger} cannot be used in general.

Nevertheless, observe that for every $\alpha,\beta>0$, we have $n_2\leq \frac{n}{2}$, and thus, the asymptotic degree sequence is well-behaved. This in particular implies that $M_{\cD_n}\geq 2m-n$. 
 
Let $\beta_0= 3.47875\dots$
be a solution to the equation $\zeta(\beta-2)-2\zeta(\beta-1)=0$. 
This is the important threshold for the appearance of the giant component
since  $\beta\geq\beta_0$ if and only if
\begin{align}\label{eq: thing}
\sum_{i\geq 1} i(i-2) \frac{n_i}{n}\leq 0\;.  
\end{align}
We refer the reader to the beginning of Section 3 in~\cite{aiello2000random} for a formal proof of this fact. 
Thus for $\beta\geq \beta_0$, the parameter $R_{\cD_n}$ is simply the maximum degree of $G$, which is $e^{\frac{\alpha}{\beta}}$. 
Since $\beta_0>1$, we have $R_{\cD_n}\ll M_{\cD_n}$.
 
\begin{theorem}[Aiello, Chung and Lu]
Let $G=G(\alpha,\beta)$ be a Power-Law random graph. Then,
\begin{enumerate}
  \item if $\beta<\beta_0$, then there exists a constant $c_1>0$ such that the probability that $G$ has a component of order at least $c_1 n$ is $1-o(1)$.
  \item if $\beta\geq\beta_0$, then for every constant $c_2>0$, the probability that $G$ has no component of order at least $c_2 n$ is $1-o(1)$.
 \end{enumerate}
 \end{theorem}
We want to emphasize that the structural description the authors give in their paper is more precise than the just mentioned result.
 
\begin{proof}
We already addressed the case $\beta\geq\beta_0$ before the theorem. 
In such a case, the sequence $\fD$ is upper-bounded and it follows from Theorem~\ref{sequencetheorem} that for every constant $c_2>0$, the  probability that $G$ has no component of order at least $c_2 n$ is $1-o(1)$.
%Since the left hand side of~\eqref{eq: thing} is strictly decreasing as a function of $\beta$, we have that $\sum_{i\geq 1} i(i-2) \frac{n_i}{n}<0$. Therefore, $j_\cD=n$ and $R_{\cD_n}= e^{\alpha/\beta} =o(M_{D_n})$ for any such $\beta$. Therefore, $\cD$ is upper-bounded and it follows from Theorem~\ref{sequencetheorem} that for every constant $c_2$, the  probability that $G$ has no component of order at least $c_2 n$, is $1-o(1)$.

Now, we consider $\beta<\beta_0$. We will show that in this case $\fD$ is lower-bounded. We split its proof into three cases,
\vspace{0.2cm}

\noindent \textbf{Case} $0<\beta\leq 2$: 
Suppose $\beta \not=2$. 
The number of edges is of order $e^{\frac{2\alpha}{\beta}}$ and thus the average degree is of order $e^{(\frac{2}{\beta}-1)\alpha}$, which tends to $\infty$ as $\alpha\to \infty$. 
Thus, provided that $n$ is large enough, we have $R_{\cD_n}\geq  \frac{M_{\cD_n}}{2}$. 
The same argument applies for $\beta=2$, as the average degree is of order $\alpha=\Theta(\log n)$.

\vspace{0.2cm}

\noindent \textbf{Case} $2<\beta<3$: Let $k$ be the smallest integer such that $\sum_{i=1}^k i^{2-\beta}>4\zeta(\beta-1)$.
Thus
\begin{align*}
	\sum_{i=1}^k i(i-2) e^{\alpha} i^{-\beta}
	\geq e^{\alpha}\left(\sum_{i=1}^k i^{2-\beta} -2\sum_{i=1}^\infty i^{1-\beta}\right)
	\geq \frac{2\zeta(\beta-1)}{\zeta(\beta)}\cdot n.
\end{align*}
%This implies that $j_\cD$, and 
Therefore,
\begin{align*}
	R_{\cD_n}\geq\sum_{i=k+1}^{e^{\frac{\alpha}{\beta}}} i \cdot\frac{e^\alpha}{i^{\beta}} 
	\geq c M_{\cD_n},
	\end{align*}
for some small constant $c=c(\beta)>0$. Here we used that $M_\cD\geq 2m-n\approx (\zeta(\beta-1)-\zeta(\beta))e^\alpha$.
\vspace{0.2cm}

\noindent \textbf{Case} $3\leq \beta<\beta_0$: Suppose now that $\beta>3$ (the case $\beta=3$ is similar to what follows). 
Let $\epsilon= \zeta(\beta-2)-2\zeta(\beta-1)$
and $k$ be the smallest integer $i$ such that $\frac{1}{(\beta-3)k^{\beta-3}}< \frac{\epsilon}{2}$. 
Using an integral approximation of the sum, we obtain $\sum_{i=1}^k i^{2-\beta}\approx \zeta(\beta-2) - \frac{1}{(\beta-3)k^{\beta-3}}$. Thus
\begin{align*}
	\sum_{i=1}^k i(i-2) e^{\alpha} i^{-\beta}
	\geq e^{\alpha}\left(\sum_{i=1}^k i^{2-\beta} -2\sum_{i=1}^\infty i^{1-\beta}\right)
	\geq \frac{\epsilon}{2\zeta(\beta)}\cdot n.
\end{align*}
As in the previous case, we have $R_{\cD_n}\geq c M_{\cD_n}$ for some constant $c=c(\beta)>0$.

\vspace{0.3cm}
Since $\fD$ is well-behaved and lower-bounded for $\beta<\beta_0$, 
we can use Theorem~\ref{sequencetheorem} to conclude that there exists a constant $c_1>0$ such that the probability that $G$ has a component of order at least $c_1 n$ is $1-o(1)$.
\end{proof}

\bibliographystyle{amsplain}
\bibliography{giant_ref}

\end{document}